\documentclass[10pt]{amsart}
\usepackage{amscd,amsmath,amssymb}
\usepackage{hyperref}
\usepackage[all]{xy}
\usepackage{graphicx}
\usepackage{psfrag}

\newcommand{\la}{\lambda}
\newcommand{\al}{\alpha}
\newcommand{\be}{\beta}

\newcommand{\f}{\varphi}

\newcommand{\vv}{\overrightarrow}

\newcommand{\Ker}{\mathop\mathrm{Ker\,}}

\newcommand{\CC}{\mathbb{C}}

\newcommand{\RR}{\mathbb{R}}

\newcommand{\NN}{\mathbb{N}}

\numberwithin{equation}{section}

\newtheorem{te}{Theorem}[section]

\newtheorem{pr}{Proposition}[section]
\newtheorem{co}{Corollary}[section]
\newtheorem{lm}{Lemma}[section]

\theoremstyle{definition}
\newtheorem{de}{Definition}[section]
\newtheorem{re}{Remark}[section]
\newtheorem{ex}{Example}[section]
\newtheorem{exs}{Examples}[section]

\begin{document}

\title{K\"ahler and Sasakian-Einstein Quotients}
\thanks{The author aknowledges financial support from the Swiss National Science Foundation, from FCT/POCTI/FEDER, and from
the grant POCI/MAT/57888/2004.}

\author{Oana Mihaela Dr\u agulete}
\address{Instituto Superior T\'ecnico, Departamento de Mat\'ematica,
Av. Rovisco Pais 1049-001, Lisbon, Portugal}
\email{odragul@math.ist.utl.pt}

\subjclass{53D05, 53D10, 53D20, 53C25, 37J05, 37J15} 
\keywords{symplectic (K\"ahler) manifolds, contact (Sasakian) manifolds, symplectic (K\"ahler) ray reduction, K\"ahler and Sasakian-Einstein quotients, conformal Hamiltonian systems, momentum map.}
\date{16-03-08}

\begin{abstract}
We construct symplectic and K\"ahler ray reduced spaces and
discuss their relation with the Marsden-Weinstein (point)
reduction. This K\"ahler reduction is well defined even when the momentum value is not totally isotropic. 
The compatibility of the ray reduction with the cone
construction and the Boothby-Wang fibration is presented. Using
the compatibility with the cone construction we provide the exact
description of ray quotients of cotangent bundles. Some
applications of the ray reduction to the study of conformal
Hamiltonian systems are described. We also give necessary and
sufficient conditions for the (ray) quotients of K\"ahler
(Sasakian)-Einstein manifolds to be again K\"ahler
(Sasakian)-Einstein. 
\end{abstract}

\maketitle

\tableofcontents

\section{Introduction}

In this paper we study geometric properties of Sasakian and
K\"ahler quotients. For manifolds endowed with a Lie group $G$ of
symmetries, we construct a reduction procedure for symplectic and
K\"ahler manifolds using the ray pre-images of the associated
momentum map $J$. More precisely, instead of taking as in point
reduction (Weinstein-Marsden reduce spaces, usually denoted by
$M_\mu$), the pre-image of a momentum value $\mu$, we take the pre-image 
of $\RR^+\mu$, the positive ray of $\mu$. And instead of
taking the quotient with respect to the isotropy group $G_\mu$ of
the momentum with respect to the coadjoint action of $G$, we take
it with respect to the kernel group of $\mu$, a normal subgroup of
$G_\mu$. The ray reduced spaces will be denoted by
$M_{\RR^+\mu}$. We have three reasons to develop this
construction.

One is geometric: the construction of non-zero, well defined
K\"ahler reduced spaces. K\"ahler point reduction is not always
well defined. The problem is that the complex structure may not
leave invariant the horizontal distribution of the Riemannian
submersion $\pi_\mu:J^{-1}(\mu)\rightarrow
M_\mu:=J^{-1}(\RR^+\mu)/G_\mu$. The solution proposed in the
literature, is based on the Shifting Theorem (see Theorem $6.5.2$
in \cite{ortega--ratiu}). More precisely, one endows the coadjoint
orbit of $\mu$, $\mathcal{O}_\mu$ with a unique up to homotheties
K\"ahler-Einstein metric of positive Ricci curvature. This
uniqueness modulo homotheties is guaranteed by the choice of an
$Ad^*$-invariant scalar product on $\mathfrak{g}^*$. Then, one
performs the zero reduction of the K\"ahler difference of the base
manifold $M$ and $\mathcal{O}_\mu$. Unfortunately, this
construction is correct only in the case of totally isotropic
momentum (i.e. $G_\mu=G$). Otherwise, using the unique
K\"ahler-Einstein form on the coadjoint orbit, instead of the
Kostant-Kirillov-Souriau form makes impossible the use of the
Shifting Theorem since the momentum map of the orbit will no
longer be the inclusion. Even so, one could take by definition the
reduced space at $\mu$ momentum to be the zero reduced space of
the symplectic difference of $M$ and $\mathcal{O}_\mu$.  But this
reduced space is not canonical, in the sense that the pull-back
through the quotient projection of the reduced K\"ahler structure
is no longer the initial one. On the other hand, the ray K\"ahler
reduction always exists and is canonical.

The second reason is that it provides invariant submanifolds for
conformal Hamiltonian systems (see \cite{mclachlan--perlmutter})
and consequently, the right framework for the reduction of
symmetries of such systems. They are usually non-autonomous
mechanical systems with friction whose integral curves preserve,
in the case of symmetries, the ray pre-images of the momentum map,
and not the point pre-images.

The third reason is finding necessary and sufficient conditions
for quotients of K\"ahler (Sasakian)-Einstein manifolds to be
again K\"ahler (Sasakian)-Einstein. Using techniques of A. Futaki
(see \cite{futaki--unu}, \cite{futaki--doi}), we prove that, under
appropriate hypothesis, ray quotients of K\"ahler-Einstein
manifolds remain K\"ahler-Einstein. We can thus construct new
examples of K\"ahler (Sasakian)-Einstein metrics.

As examples of symplectic (K\"ahler) and contact (Sasakian) ray
reductions we treat the case of cotangent and cosphere bundles. We
show, proving a shifting type theorem that, theoretically,
$(T^*Q)_{\RR^+\mu}$ and $(S^*Q)_{\RR^+\mu}$ are universal ray
reduced spaces. Concrete examples of toric actions on spheres are
also computed.

The paper is structured as follows. Section $2$ presents the
symplectic and K\"ahler ray reduction treating separately the case
of exact symplectic manifolds. In the fourth section of this paper
we deal with the cone and Boothby-Wang compatibilities with the
ray reduction. We show that the ray
reduction of the cone of a contact manifold is exactly the cone of
the contact reduced space. As a corollary we obtain the ray
reduction of cotangent bundles. Also, we prove that the
Boothby-Wang fibration associated to a quasi-regular, compact,
Sasakian manifold descends to a Boothby-Wang fibration of the ray
reduced spaces. Section $4$ presents the study of conformal
Hamiltonian systems. We extend the class of conformal Hamiltonian
systems already studied in the literature and we complete the
existing Lie Poisson reduction with the general ray one, making
thus use of the conservative properties of the momentum map. We illustrate all
these with the example of a certain type of Rayleigh systems.
We also give a characterization of relative equilibria for this
type of systems. In Section $5$ we perform the ray reduction of
cotangent bundles of Lie groups, as well as the reduction of their
associated cosphere bundles. We show, proving a shifting type
theorem that, theoretically, $(T^*G)_{\RR^+\mu}$ and
$(S^*G)_{\RR^+\mu}$ are universal ray reduced spaces. The role of
the coadjoint orbit of the ray momentum $\RR^+\mu$ in the
construction of these universal reduced spaces is made clear. In
the last section we find necessary and sufficient conditions for
the ray reduced space of a K\"ahler-Einstein manifold of positive
Ricci curvature to be again K\"ahler-Einstein. Using the
compatibility of ray reduction with the Boothby-Wang fibration, we
obtain as a corollary similar conditions for the Sasakian-Einstein
case. All these are illustrated with concrete examples in which we construct
new K\"ahler (Sasakian)-Einstein manifolds.

\section{Symplectic and K\"ahler Ray-Reductions}
\label{sectiunea--trei}

Let $G$ be a Lie group acting smoothly, properly, by
symplectomorphisms and in a Hamiltonian way on a symplectic manifold
$(M,\omega)$. Denote by $J:M\rightarrow\mathfrak{g}^*$ the
associated momentum map and recall that it is $G$-equivariant. For
any element $\mu\in\mathfrak{g}^*$, let $K_\mu$ be the unique
connected, normal Lie subgroup of $G_\mu$ with Lie algebra given by
$\mathfrak{k}_\mu = \ker\: (\mu| _{\mathfrak{g}_\mu})$. This group
is called the \textit{kernel group of $\mu$}.
\begin{de}
We define the quotient of $M$ by $G$ at $\RR^+\mu$ to be
$M_{\RR^+\mu}:=J^{-1}(\RR^+\mu)/K_\mu$. $M_{\RR^+\mu}$ will be
called the \textit{ray reduced space at} $\mu$.
\end{de}

In a paper of Guillemin and Sternebrg (\cite{guillemin--sternberg}, Example $4$) we found a 
geometric interpretation for the kernel algebra of $\mu$.
Let $\mathcal{O}_{\RR^+\mu}$ be the \textit{cone of the coadjoint orbit through $\mu$} defined by
\begin{equation}
\label{conul--orbitei--coadj}
\mathcal{O}_{\RR^+\mu}:=\{Ad^*_{g^{-1}}r\mu\,|\,g\in G\,\text{,}\,
r\in\RR^+\}.
\end{equation}
The conormal space at $\mu$ of $\mathcal{O}_{\RR^+\mu}$ is precisely $\mathfrak{k}_\mu$. This can be easily deduced
using the characterization of the tangent space at $\mu$ of $\mathcal{O}_{\RR^+\mu}$ given in Proposition \ref{sptanglaorbita}.

In this section we will show that, under certain hypothesis, the
ray quotient admits a natural symplectic or K\"ahler structure,
once the initial manifold is symplectic or K\"ahler. The proof of
the next theorem is an analogous of the proof given in
\cite{willett} for the contact case (see Theorem $1$). As all
reduction theorems, it mainly uses arguments in linear symplectic
or contact algebra.

For the two results of this section we will need three lemmas. The
first is a characterization of a locally free action and the last
two are classical results of symplectic linear algebra.
\begin{lm}
\label{lema1} $J$ is transverse to $\RR^+\mu$ if and only if $K_\mu$
acts locally freely on $J^{-1}(\RR^+\mu)$.
\end{lm}

\begin{lm}
\label{lema2}
 Consider a symplectic vector space $(V,\Omega)$
and $W\subset V$ an isotropic subspace. Then,
$\operatorname{ker}\Omega\mid_{W^\Omega}=W$, where $W^\Omega$ is the
symplectic perpendicular of $W$.
\end{lm}

\begin{lm}
\label{lema3} Let $V$ be a vector space and $\Omega:V\times
V\rightarrow\RR$ an antisymmetric and bilinear two-form. If $V$
admits the direct decomposition $V=X\oplus V$ with respect to
$\Omega$ and
$\operatorname{ker}\Omega\subseteq\operatorname{ker}\Omega\mid_X$,
then $\operatorname{ker}\Omega=\operatorname{ker}\Omega\mid_X$.
\end{lm}

We are now ready to prove the first theorem of this section.

\begin{te}
\label{reducere--unu}
 Suppose $(M, \omega)$ is a symplectic manifold endowed with a
 Hamiltonian action of the Lie group $G$. Let $\mu\in\mathfrak{g}^*$
 and $K_\mu$ its
kernel group. Denote by $J:M\rightarrow\mathfrak{g}^*$ the
associated momentum map and assume that the following hypothesis are
verified:

$1^\circ$ $K_\mu$ acts properly on $J^{-1}(\RR^+\mu)$;

$1^\circ$ $J$ is transverse to $\RR^+\mu$;

$3^\circ$ $\mathfrak{g}=\operatorname{ker}\mu+\mathfrak{g}_\mu$.

Then the ray quotient  at $\mu$
$$
M_{\RR^+\mu}:=J^{-1}(\RR^+\mu)/K_\mu
$$
is a naturally symplectic orbifold, i.e. its symplectic structure
$\omega_{\RR^+\mu}$ is given by
\begin{equation*}
\label{doamne--ajutor}
\pi^*_{\RR^+\mu}\omega_{\RR^+\mu}=i^*_{\RR^+\mu}\omega,
\end{equation*} where
$$\pi_{\RR^+\mu}:J^{-1}(\RR^+\mu)\rightarrow M_{\RR^+\mu}\qquad\text{and}\qquad
 i_{\RR^+\mu}:J^{-1}(\RR^+\mu)\hookrightarrow M$$ are the canonical projection
and immersion respectively.
\end{te}

\begin{proof}
The transversality of the momentum map with respect to $\RR^+\mu$,
ensures that $J^{-1}(\RR^+\mu)$ is a submanifold of $M$. Lemma
\ref{lema1} implies that the quotient $M_{\RR^+\mu}$ is an
orbifold and that $\pi_{\RR^+\mu}$ is a surjective submersion in
the category of orbifolds.

The first step is to see that the restriction of the symplectic form
on $J^{-1}(\RR^+\mu)$ is projectable on the quotient $M_{\RR^+\mu}$. 
For any $\xi\in
\mathfrak{k}_\mu$ and any $x$ in $M$, we have that
$$
T_x\pi_{\RR^+\mu}(\xi_M(x))=\left.\frac{d}{dt}\right|_{t=0}\pi_{\RR^+\mu}(\operatorname{exp}t\xi\cdot
x)=\left.\frac{d}{dt}\right|_{t=0}\pi_{\RR^+\mu} (x)=0.
$$
Hence, $\langle\{\xi_{J^{-1}(\RR^+\mu)}\,|\,
\xi\in\mathfrak{k}_\mu \}\rangle
\subset\operatorname{ker}(T\pi_{\RR^+\mu})$. A count of dimensions
shows that, in fact, the vertical distribution of $\pi_{\RR^+\mu}$
is generated by all the infinitesimal isometries associated to the
elements of $\mathfrak{k}_\mu$. Since
$\omega\mid_{J^{-1}(\RR^+\mu)}=i_{\RR^+\mu}^*\omega$ is
$K_\mu$-invariant, it follows that its Lie derivative with respect
to all vector fields $\{\xi_{J^{-1}(\RR^+\mu)}\,|\,
\xi\in\mathfrak{k}_\mu \}$ is zero. Let $x\in J^{-1}(\RR^+\mu)$
with $J(x)=r\mu$ and $v\in T_x(J^{-1}(\RR^+\mu))$. Then,
identifying $T_{J(x)}\RR^{+}\mu$ with $\RR\mu$, we obtain
\begin{equation*}
\begin{split}
\omega(i_{\RR^+\mu}(x))(\xi_M(x), T_x i_{\RR^+\mu}
v)=T_{i_{\RR^+\mu}(x)}J\mid_{J^{-1}(\RR^+\mu)}(v)(\xi)=\\
i_{\RR^+\mu}^*(TJ\mid_{J^{-1}(\RR^+\mu)})(v)(\xi)=\mu(\xi)=0.
\end{split}
\end{equation*}
It follows that $i_{\RR^+\mu}^* \omega$ is a basic two-form which
projects on $M_{\RR^+\mu}$ to the closed form
$\omega_{\RR^+\mu}\in \Lambda^2 (T^* M_\mu)$ with the property
that $\pi^*_{\RR^+\mu}\omega_{\RR^+\mu}=i^*_{\RR^+\mu}\omega$.

Since $\omega_{\RR^+\mu}$ is a closed form, it remains to prove
that it is also non-degenerate. For this, we will show that
$T_x(K_\mu\cdot x)=\ker (i^*_{\RR^+\mu}\omega)(x)$, for any $x\in
J^{-1}(\RR^+\mu)$. Fix $x\in J^{-1}(\RR^+\mu)$ with $J(x)=t\mu$
and denote by $\Psi:M\rightarrow \mathfrak{k}_\mu^*$ the momentum
map associated to the action of the kernel group of $\mu$ on $M$.
Let $i^T:\mathfrak{g}^*\hookrightarrow\mathfrak{k}^*_\mu$ be the
canonical inclusion. Then, $\Psi=i^T\circ J$ and
$J^{-1}(\RR^+\mu)\subset
J^{-1}(\mathfrak{k}^\circ_\mu)=\Psi^{-1}(0)$. Notice that
$J^{-1}(\RR^+\mu)\cap G\cdot x=G_{\RR^+\mu}\cdot x$, where
$G_{\RR^+\mu}=\{g\in G\mid \operatorname{Ad}^*_g \mu=r\mu, r>0\}$
is the ray isotropy group of $\mu$. This Lie group has many interesting
properties for which we refer the reader to Section
\ref{Sectiunea--sase}. 

For any $v\in (T_x (K_\mu \cdot
x))^{\omega_x}$, $\omega_x (v,\xi_M (x))=0$,
$\forall\xi\in\mathfrak{k}_\mu$ if and only if $T_x J(v)(\xi)=0$,
$\forall\xi\in\mathfrak{k}_\mu$. Therefore, $(T_x (K_\mu\cdot
x))^{\omega_x}=T_x U$, where
$U:=J^{-1}(\mathfrak{k}^{\circ}_\mu)=\Psi^{-1}(0)$. We can assume
$U$ to be a submanifold of $M$ because the transversality
condition satisfied by the momentum map implies that $K_\mu$ acts
locally freely at least on a neighborhood of $J^{-1}(\RR^+\mu)$ in
$U$, if not on the whole $U$.

Applying Lemma \ref{lema2} for $(V,\Omega):=(T_x M,\omega_x )$ and
$W:=T_x (K_\mu\cdot x)$, we obtain that
$\operatorname{ker}\omega_x \mid_{T_x U}=T_x (K_\mu\cdot x)$. We
have already seen that $T_x (K_\mu\cdot
x)\subset\operatorname{ker}i^*_{\RR^+\mu}\omega_x$. It follows
that
\begin{equation}
\label{eq3} \operatorname{ker}\omega_x \mid_{T_x
U}\subset\operatorname{ker}\omega_x\mid_{T_x J^{-1}(\RR^+\mu)}.
\end{equation}

Since $\mathfrak{g}=\operatorname{ker}\mu+\mathfrak{g}_\mu$, we can
chose a decomposition
\begin{equation}
\label{descompunere--m}
\mathfrak{g}=\mathfrak{g}_\mu\oplus\mathfrak{m}\quad\text{with}\quad
\mu\mid_\mathfrak{m}=0.
\end{equation}
 Let $\mathfrak{m}_M:=\{\xi_M(x)\mid\xi\in
M\}$. For any $\xi\in\mathfrak{m}$ and $\eta\in\mathfrak{k}_\mu$,
the equivariance of the momentum map implies that
$$
T_xJ(\xi_M(x))(\eta)=\xi_{\mathfrak{g}^*}(t\mu)(\eta)=
-t\langle\mu,[\xi,\eta]\rangle=t\eta_{\mathfrak{g}^*}(\mu)(\xi)=0.
$$
Therefore, $\mathfrak{m}_M(x)\subset T_xU$ and
$T_xJ(\mathfrak{m}_M(x))\subset T_{t\mu}(G\cdot t\mu)$. It is easy
to see that
$$
T_xJ\mid_{\mathfrak{m}_M(x)}:\mathfrak{m}_M(x)\rightarrow
T_{t\mu}(G\cdot t\mu)
$$
is a linear isomorphism and, hence,
\begin{equation}
\label{eq1}
 T_xJ(\mathfrak{m}_M(x))=T_{t\mu}(G\cdot t\mu).
\end{equation}
Notice that equation \eqref{eq1}, the third hypothesis of the theorem which can
equivalently be expressed as
$\{0\}=(\operatorname{ker}\mu)^\circ\cap(\mathfrak{g}_\mu)^\circ=\RR\mu\cap
T_{t\mu}(G\cdot t\mu)$, and the fact that
$T_xJ(J^{-1}(\RR^+\mu))\subset\RR\mu$ imply that
\begin{equation}
\label{eq2} \mathfrak{m}_M(x)\cap T_xJ^{-1}(\RR^+\mu)=\{0\}.
\end{equation}
A simple dimension calculus shows that $\mathfrak{m}_M(x)$ and
$T_xJ^{-1}(\RR^+\mu)$ are complementary subspaces of $T_xU$. We
have also seen that they are perpendicular with respect to
$\omega_x\mid_{T_xU}$. Using relation \eqref{eq3}, we can now
apply Lemma \ref{lema3} for $V:=T_xM$, $W:=\mathfrak{m}_M(x)$, and
$X:=T_xJ^{-1}(\RR^+\mu)$. Thus, we obtain that
$\operatorname{ker}\omega_x\mid_{T_xU}=T_x(K_\mu\cdot
x)=\operatorname{ker}\omega_x\mid_{T_xJ^{-1}(\RR^+\mu)}$, for any
$x\in J^{-1}(\RR^+\mu)$. This shows that $\omega_{\RR^+\mu}$ is a
non-degenerate form, completing thus our proof.
\end{proof}

Notice that in the case $\mu=0$ we recover the reduced symplectic
space at zero. Without the hypothesis that $K_\mu$ acts properly
on $J^{-1}(\RR^+\mu)$, the quotient $M_{\RR^+\mu}$ may not be
Hausdorff. As the Lemma \ref{lema1} proves, the second hypothesis
of this theorem ensures that $M_{\RR^+\mu}$ is an orbifold. If
$\mu$ is non-zero and the kernel and isotropy groups of $\mu$
coincide, then the quotient may fail to be symplectic. This always happens 
when the coadjoint orbit of $\mu$ is nilpotent (i.e. $\mathcal{O}_{\mu}=\mathcal{O}_{\RR^+\mu}$).
As an example, consider the cotangnet lift of the action of $SL(2,\RR)$ on 
itself by left translations. Identifying $T^*(SL(2,\RR))$ with $SL(2,\RR)\times sl(2,\RR)^*$ and taking 
as momentum value $\mu:=\langle \left(\begin{matrix}0 & 1\\ 0 & 0
\end{matrix}\right),\cdot\rangle$, one can check that $\mathcal{O}_{\mu}=\mathcal{O}_{\RR^+\mu}$.
Even more, $\ker{\mu}=\{\left(\begin{matrix}\alpha & 0\\ \gamma & -\alpha
\end{matrix}\right)|\alpha, \gamma\in\RR\}\supset\mathfrak{g}_\mu=\{\left(\begin{matrix}0 & 0\\ \gamma & 0
\end{matrix}\right)|\gamma\in\RR\}$. Except the last one, all 
the hypothesis of Theorem \ref{reducere--unu} are fulfilled.
Since the cotangent action is free, and $\operatorname{dim} J^{-1}(\RR^+\mu)=4$, 
$K_\mu=\{\left(\begin{matrix}1 & 0\\ t & 1 \end{matrix}\right)|t\in\RR\}$, the quotient $J^{-1}(\RR^+\mu)/K_\mu$ is $3$-dimesnional, and hence not symplectic.

\begin{co}
In the hypothesis of Theorem \ref{reducere--unu}, if the dimension
of $M$ is $2n$ and the Lie group $G$ is $d$-dimensional, then the
dimension of the symplectic quotient is $2n-2k-m=2n-p-d+2$, where
$p=\operatorname{dim}G_\mu=k+1$.
\end{co}

In the symplectic point reduction, the reduced spaces of exact
manifolds are not always exact. This is, however true, only if one
performs reduction at zero momentum. Recall, for instance, that
coadjoint orbits which are point reduced spaces are not
necessarily exact symplectic manifolds. A counter example may be
found in \cite{marsden--ratiu}, Example $(a)$ of Section $14.5$.
Surprisingly, ray quotients of exact symplectic manifolds are
exact for any momentum.

\begin{co}
In the hypothesis of Theorem \ref{reducere--unu}, if $(M,
\omega)=(M,-d\theta)$ with $\theta$ a $K_\mu$-invariant one form,
then the ray quotient will also be exact.
\end{co}

\begin{proof}
We want to show that $i^*_{\RR^+\mu} \theta$ is a basic form for
the projection $\pi_{\RR^+\mu}: J^{-1}(\RR^+\mu)\rightarrow
M_{\RR^+\mu}$. The $K_\mu$-invariance ensures that
$L_{\xi_{J^{-1}(\RR^+\mu)}} i^*_{\RR^+\mu}\theta=0$, for any $\xi$
in the kernel algebra of $\mu$. For $x\in J^{-1}(\RR^+\mu)$, we
have that
$$
i^*_{\RR^+\mu}\theta(i_\mu(x))(\xi_{J^{-1}(\RR^+\mu)}(x))=J(x)(\xi)=r\mu(\xi)=0.
$$

Hence, $i_{\xi_{J^{-1}(\RR^+\mu)}}(i^*_{\RR^+\mu} \theta )=0$, for
any $\xi\in\mathfrak{k}_\mu$, proving that $i^*_{\RR^+\mu} \theta$
is basic. Therefore, there is a one form $\theta_{\RR^+_\mu}$ such
that $i^*_{\RR^+\mu} \theta=\pi^*_{\RR^+\mu} \theta$. Using
Theorem \ref{reducere--unu}, we get that
$$
\pi^*_{\RR^+\mu} (-d\theta_{\RR^+\mu})=d(-\pi^*_{\RR^+\mu}
\theta_{\RR^+\mu})=-di^*_{\RR^+\mu}
\theta=i^*_{\RR^+\mu}(-d\theta)=i^*_{\RR^+\mu}
\omega=\pi^*_{\RR^+\mu} \omega_{\RR^+\mu}.
$$
Since $\pi^*_{\RR^+\mu}$ is injective we obtain that
$\omega_{\RR^+\mu}=-d\theta_{\RR^+\mu}$.
\end{proof}

A large class of examples can be obtained in the case when $(M,
\omega)$ is the cotangent bundle of a manifold $Q$ endowed with
the canonical symplectic form $\omega_0=-d\theta_0$. We treat
this case in Section \ref{sectiunea--patru}, Corollary
\ref{red--raza--cotang}.

 We will now extend this reduction
procedure to the metric context, i.e. for K\"ahler manifolds.

\begin{te}
\label{reducere--doi} Let $(M,\operatorname{g}, \omega)$ be a K\"ahler manifold
and $G$ a Lie group acting on $M$ by Hamiltonian
symplectomorphisms. If $J:M\rightarrow\mathfrak{g}^*$ is the
momentum map associated to the action of $G$ and $\mu$ an element
of $\mathfrak{g}^*$, assume that:

$1^\circ$ $\Ker\mu+\mathfrak{g}_\mu=\mathfrak{g}$;

$2^\circ$ the action of $K_\mu$ on $J^{-1}(\RR^+\mu)$ is proper and
by isometries;

$3^\circ$ $J$ is transverse to $\RR^+\mu$.

Then the ray quotient at $\mu$
$$M_{\RR^+\mu}:=J^{-1}(\RR^+\mu)/K_\mu$$ is a K\"ahler orbifold with respect
to the projection of the metric $\operatorname{g}$.
\end{te}

\begin{proof}
From Theorem \ref{reducere--unu}, we already know that
$(M_{\RR^+\mu}, \omega_{\RR^+\mu})$ is a symplectic orbifold. It
remains to show that the symplectic structure is also a K\"ahler
one with corresponding metric given by the projection of
$\operatorname{g}$. The second hypothesis of the theorem ensures that
$(J^{-1}(\RR^+\mu)\,,\, i_{\RR^+\mu}^* \operatorname{g})$ is an
isometric Riemannian submanifold of $M$.

Again, we will use a decomposition
$\mathfrak{g}=\mathfrak{g}_\mu\oplus\mathfrak{m}$, where
$\mu\mid_\mathfrak{m}=0$. Let $\Psi:M\rightarrow \mathfrak{k}_\mu^*$
be the momentum map associated to the action of $K_\mu$ on $M$ and
$\mathfrak{m}_M:=\{\xi_M(x)\mid\xi\in M\}$. In the proof above we
have already seen that
\begin{equation}
\label{eq4}
T_xJ^{-1}(\RR^+\mu)\oplus\mathfrak{m}_M(x)=T_x\Psi^{-1}(0),
\end{equation}
 for
any $x\in J^{-1}(\RR^+\mu)$. Let $\{\xi_1,\cdots\xi_k\}$ and
$\{\eta_1,\cdots\eta_m\}$ be basis in $\mathfrak{k}_\mu$ and
$\mathfrak{m}$ respectively, with $m=\dim\mathfrak{m}$ and
$k=\dim\mathfrak{k}_\mu$. Without loss of generality, we can assume
that the infinitesimal isometries $\{\xi_{iM}\}_{i=1, k}$ and
$\{\eta_{jM}\}_{j=1, m}$ are $\operatorname{g}$-orthogonal. Thus, $\{J\xi_{iM}, J\eta_{jM}\}_{i, j}$ are
linearly independent in each point of $J^{-1}(\RR^+\mu)$. Even more,
$\{J\xi_{iM}, J\eta_{jM}\}_{i, j}$ belong to the normal fiber bundle
of $J^{-1}(\RR^+\mu)$ since
$$
g(J\eta_{jM}, V)=g(J\xi_{iM}, V)=\omega(\xi_{iM},
V)=-TJ(V)(\xi)=-r\mu(\xi_i)=-r\mu(\eta_j)=0
$$
for any $V$ vector field on $J^{-1}(\RR^+\mu)$. The next step is to
show that $\{J\xi_{iM}\}_{i=1, k}$ is a basis in the normal bundle
of $T\Psi^{-1}(0)$. Notice that
$\{\xi_{iM}\mid_{J^{-1}(\RR_+\mu)}\}_{i=1,k}$ are tangent to
$J^{-1}(\RR^+\mu)$ and
$$
\operatorname{g}(J\xi_{iM},
V)=\omega(\xi_{iM},V)=T\Psi(V)(\xi_i)=Ti^{T}(TJ(V))(\xi_i)=0,
$$
for any $V$ differentiable section of $T\Psi^{-1}(0)$. Here, we
have used that $\Psi=i^*_T\circ J$, where
$i_T^*:\mathfrak{g}^*\rightarrow \mathfrak{k}^*_\mu$ is the
canonical projection. Therefore, $\{J\xi_{iM}\}_{i=1, k}$ are
vector fields normal to $TU$, where
$U=J^{-1}(k^\circ_\mu)=\Psi^{-1}(0)$. As $\dim TU=\dim
M-\dim\mathfrak{k}_\mu$, these vector fields form a basis of the
normal fiber bundle to $TU$. Equation \eqref{eq4} implies that
$\{J\xi_{iM}, J\eta_{jM}\}_{i, j}$ form a basis of the normal
bundle to $J^{-1}(\RR^+\mu)$. Since the action of $K_\mu$ on
$J^{-1}(\RR^+\mu)$ is isometric, $i_{\RR^+\mu}^*\operatorname{g}$
projects on $M_{\RR^+\mu}$ in $\operatorname{g}_{\RR^+\mu}$ and
the projection $\pi_{\RR^+\mu}$ becomes thus a Riemannian
submersion. Obviously, the vertical distribution of this
Riemannian submersion is given by $\{\xi_{iM}\}_{i=1,k}$. Then,
$T_xJ^{-1}(\RR^+\mu)=\{\xi_{iM}\}(x)\oplus \mathcal{H}_x$, where
$\mathcal{H}_x$ is the horizontal distribution at $x$ associated
to the Riemannian submersion  $\pi_\mu$. To see that
$(\omega_{\RR^+\mu},\operatorname{g}_{\RR^+\mu})$ is an almost
K\"ahler structure, we need to check that
$$
\omega_{\RR^+\mu}([x])(T_x\pi_\mu v,T_x\pi_\mu
w)=\operatorname{g}_{\RR^+\mu}([x])(\mathcal{C}_{\RR^+\mu} T_x\pi_\mu v,T_x\pi_\mu
w),
$$
for any $[x]=\pi_\mu(x)\in J^{-1}(\RR^+\mu)$ and $v,w\in
\mathcal{H}_x$. Here, $\mathcal{C}_{\RR^+\mu}$ denotes the
projection of the complex structure $\mathcal{C}$ of $\omega$.
Since $T_x\pi_\mu$ is an isomorphism from the horizontal space at
$x$ onto $T_{[x]}M_{\RR^+\mu}$ which identifies
$(\omega_{\RR^+\mu},\operatorname{g}_{\RR^+\mu})([x])$ with
$(i_{\RR^+\mu}^*\omega,i_{\RR^+\mu}^*\operatorname{g})\mid_{\mathcal{H}_x}$
suffices to show that the horizontal distribution is
$\mathcal{C}$-invariant. Let $v\in \mathcal{H}_x$. Then
$\omega(\mathcal{C}v, \xi_{iM})=\operatorname{g}(v, \xi_{iM})=0$,
for any $\xi_i\in\mathfrak{k}_\mu$. Also
$\operatorname{g}(\mathcal{C}v,\mathcal{C}\xi_{iM})=\operatorname{g}(v,\xi_{iM})=0$
and
$\operatorname{g}(\mathcal{C}v,\mathcal{C}\eta_{jM})=\operatorname{g}(v,\eta_{jM})=0$,
for all $i=1,k$ and $j=1,m$. It follows that $\mathcal{C}v$ is
also a horizontal vector. To show that $\mathcal{C}_{\RR^+\mu}$ is
integrable we will evaluate the Nijenhuis tensor $N_{\RR^+\mu}$. Thus,
\begin{align*}
N_{\RR^+\mu}(T_x\pi_\mu(v), T_x\pi_\mu(w))
 =[T_x\pi_\mu( v),T_x\pi_\mu(w)]-[\mathcal{C}_{\RR^+\mu}
 T_x\pi_\mu(v),\mathcal{C}_{\RR^+\mu}
T_x\pi_\mu(w)]&\\+\mathcal{C}_{\RR^+\mu}([\mathcal{C}_{\RR^+\mu}
T_x\pi_\mu(v),T_x\pi_\mu(w)])+\mathcal{C}_{\RR^+\mu}([T_x\pi_\mu(v),\mathcal{C}_{\RR^+\mu}
T_x\pi_\mu(w)])&\\
=T_x\pi_\mu([v,w])-T_x\pi_\mu
([\mathcal{C}v,\mathcal{C}w])+\mathcal{C}_{\RR^+\mu}(T_x\pi_\mu([\mathcal{C}v,w]))+\mathcal{C}_{\RR^+\mu}(T_x\pi_\mu([v,\mathcal{C}w]))
\\=
T_x\pi_\mu([v,w]-[\mathcal{C}v,\mathcal{C}w])+T_x\pi_\mu(\mathcal{C}([\mathcal{C}v,w]))+T_x\pi_mu(\mathcal{C}([v,\mathcal{C}w]))
&\\= T_x\pi_\mu(N(v,w))=0,
\end{align*}
where $N$ is the Nijenhuis tensor of $(\omega,g)$. Thus,
$\mathcal{C}_{\RR^+\mu}$ is integrable and \newline
$(M_{\RR^+\mu},\omega_{\RR^+\mu},\operatorname{g}_{\RR^+\mu})$ a K\"ahler
manifold.
\end{proof}

\begin{re}
Unfortunately, non zero K\"ahler regular point reduction is not
canonical. As it is very well explained in \cite{bryant} (see
Exercise $3$), the complex structure may not leave invariant the
horizontal distribution of the Riemannian submersion given by the
quotient projection ($\pi_{\mu}:M\rightarrow M_\mu$). Therefore it
is not projectable on $M_\mu$. The solution proposed in the
literature, is based on the Shifting Theorem (see Theorem $6.5.2$
in \cite{ortega--ratiu}). More precisely, one endows the coadjoint
orbit of $\mu$, $\mathcal{O}_\mu$ with a unique up to homotheties
K\"ahler-Einstein metric of positive Ricci curvature. For the
construction of this metric, see
\cite{kostant--doi}, Chapter $8$ in \cite{besse}, and
\cite{kirillov--doi}. This uniqueness modulo homotheties is
guaranteed by the choice of an $Ad^*$-invariant scalar product on
$\mathfrak{g}^*$. Then, one performs the zero reduction of the
K\"ahler difference of the base manifold $M$ and
$\mathcal{O}_\mu$. Unfortunately, this construction is correct
only in the case of totally isotropic momentum (i.e. $G_\mu=G$).
Otherwise, using the unique K\"ahler-Einstein form on the
coadjoint orbit, instead of the Kostant-Kirillov-Souriau form
makes impossible the use of the Shifting Theorem since the
momentum map of the orbit will no longer be the inclusion. Even
so, one could take by definition the reduced space at $\mu$
momentum to be the zero reduced space of the symplectic difference
of $M$ and $\mathcal{O}_\mu$.  But this reduced space is not
canonical, in the sense that  the pull-back through the quotient
projection of the reduced K\"ahler structure is no longer the
initial one. On the other hand, the ray K\"ahler reduction always
exists and is canonical.
\end{re}

\section{Cone and Boothby-Wang Compatibilities}
\label{sectiunea--patru} 

Traditionally, Sasakian manifolds where
defined via contact structures by adding a Riemannian metric with
certain compatibility conditions.
\begin{de}
A Sasakian structure on an exact contact manifold
$(S,\eta,\mathfrak{R})$ is a Riemannian metric $\operatorname{g}$
on $S$ such that  there is a $(1,1)$-tensor field $\Phi$ witch verifies the following identities
$$
\Phi^2=-Id+\eta\otimes\mathfrak{R}\quad\quad\eta(X)=\operatorname{g}(X,\mathfrak{R})\quad\quad
d\eta(X,Y)=\operatorname{g}(X,\Phi Y),
$$
for any vector fields $X$, $Y$.
\end{de}
  A good reference for this point of view is the book of D. E.
Blair, \cite{blair}.

There are other equivalent definitions of a Sasakian manifold and
in the following proposition we present four of them. The first
one is most in the spirit of the original definition of Sasaki
(see \cite{sasaki--unu}). The most geometric approach is
highlighted in the second definition. It only uses the holonomy
reduction of the associated cone metric and it was introduced by
C. P. Boyer and K. Galicki in \cite{boyer--galicki2}.

\begin{pr}
Let $(S,\operatorname{g})$ be a Riemannian manifold of dimension
$m$, $\nabla$ the associated Levi-Civita connection, and $R$ the
Riemannian curvature tensor of $\nabla$. Then, the following
statements are equivalent:
\begin{itemize}
\item there exists a unitary Killing vector field $\mathfrak{R}$
on $S$ so that the tensor field $\Phi$ of type $(1,1)$, defined by
$\Phi(X)=\nabla_X\mathfrak{R}$, satisfies the condition
$$
(\nabla_X\Phi)(Y)=\operatorname{g}(\mathfrak{R},Y)X-\operatorname{g}(X,Y)\mathfrak{R},
$$
for any pair of vector fields $X$ and $Y$ on $S$; 
\item 
the
holonomy group of the cone metric on $S$, $(\mathcal{C}(S),
\mathcal{C}(\operatorname{g})):=(S\times\RR^+,
r^2\operatorname{g}+dr^2)$ reduces to a subgroup of
$U(\frac{m+1}{2})$. In particular, $m=2n+1$, for a $n\ge 1$ and
$(\mathcal{C}(S), \mathcal{C}(\operatorname{g}))$ is K\"ahler;
\item there exists a unitary Killing vector field $\mathfrak{R}$
on $S$ so that the Riemannian curvature satisfies the condition
$$
R(X,\mathfrak{R})Y=\operatorname{g}(\mathfrak{R},Y)X-\operatorname{g}(X,Y)\mathfrak{R},
$$
for any pair of vector fields $X$ and $Y$ on $S$; \item there
exists a unitary Killing vector field $\mathfrak{R}$ on $S$ so
that the sectional curvature of every section containing
$\mathfrak{R}$ equals one; \item $(S,\operatorname{g})$ is a
Sasakian manifold.
\end{itemize}
\end{pr}

For the proof, see \cite{boyer--galicki2}.

\noindent \textbf{Example: Sasakian spheres.} One of the simplest
compact examples of Sasakian manifolds is the standard sphere
$S^{2n+1}\subset\CC^n$ with the metric induced by the flat one on
$\CC^n$. The characteristic Killing vector field (i.e. the
associated Reeb vector field) is given by
$\mathfrak{R}(p)=-i\vv{p}$, $i$ being the imaginary unit. The
contact form is given by $\eta:=\frac{1}{2}(dz-\sum^n y^jdx_j)$,
if $(x_j,y^j,z)_{j=1,n}$ are the canonical coordinates on the base
space.


Recall that if $(M,\eta)$ is a $2n+1$-dimensional exact contact
manifold, its symplectic cone is given by
$\mathcal{C}(M):=(M\times\RR^+ , d r^2\wedge\eta+r^2 d\eta)$ and
$M$ can be embedded in the cone as $M\times \{1\}$. The cone of a
Sasakian  $(S,\operatorname{g})$ manifold admits a canonical
K\"ahler structure given by $\mathcal{C}(\operatorname{g}):=r^2
\operatorname{g}+dr^2$. If a Lie group $G$ acts by contact
isometries on $S$, then this action can be lifted to the K\"ahler
cone as $g\cdot(x,r):=(g\cdot x,r)$, for any $g\in G$ and
$(x,r)\in\mathcal{C}(S)$. This action commutes with the
translations on the $\RR^+$ component and, in the Sasakian case,
it is by holomorphic isometries. In the Sasakian case, we can also
define a complex structure given as follows:
$$
\mathcal{C}Y:=\varphi Y-\eta(Y)R,\text{,}\, \mathcal{C}R:=\xi,
$$
where $R=r\partial_r$ is the vector field generated by the
$1$-group of transformations $\rho_t: (x,r)\rightarrow (x,tr)$ and
$\f:=\nabla\xi$, with $\nabla$ the Levi-Civita connection
associated to $\operatorname{g}$. It is easy to see that
$(S,\eta,\operatorname{g})$ is Einstein if and only if the cone
metric $\mathcal{C}(\operatorname{g)}$ is Ricci flat, i. e.,
$(\mathcal{C}(S), \mathcal{C}(\operatorname{g)})$ is Calabi-Yau
(i. e. K\"ahler Ricci-flat).

Let $\Phi:S\rightarrow\mathfrak{g}^*$ be the contact momentum map
associated to the $G$-action on $S$. The lifted action on the cone
is Hamiltonian and a corresponding equivariant symplectic momentum
map is given by
$$\Phi_s:\mathcal{C}(S)\rightarrow\mathfrak{g}^*\,\text{,}\,\Phi_s(x,r):=e^
s J(x)\,\text{,}\,\,\text{for any}\,(x,r)\in\mathcal{C}(S).$$

Having established the above notations, we are ready to prove that
reduction and the cone construction are commuting operations.

\begin{lm}
\label{lema4}
 Let $(S,\eta,\operatorname{g},\xi)$ be a Sasakian
manifold and $(\mathcal{C}(S),\mathcal{C}(\operatorname{g}), J)$
its K\"ahler cone. Suppose a Lie group $G$ acts on $S$ by strong
contactomorphisms and commuting with the action of the
$1$-parameter group generated be the field $R$. Let $\mu$ be an
element of the dual of the Lie algebra of $G$. Then the K\"ahler
cone of the reduced contact space at $\mu$ is the reduced space at
$\mu$ for the lifted action on $\mathcal{C}(S)$.
\end{lm}
\begin{proof}
Let $K_\mu$ be the kernel group of $\mu$, $(S_{\RR^+\mu}, \eta_{\RR^+\mu}, \operatorname{g}_{\RR^+\mu})$ the
corresponding contact reduced space, and
$\mathcal{C}(S_{\RR^+\mu})$ the reduced space for the lift of the
action on the cone. Since the $K_\mu$-action commutes with
homotheties on the $\RR^+$ component, there is a natural
diffeomorphism between $\mathcal{C}(S_{\RR^+\mu})$ and
$\mathcal{C}(S)_{\RR^+\mu}$: $$
\Psi:\mathcal{C}(S)_{\RR^+\mu}\rightarrow\mathcal{C}(S_{\RR^+\mu})\,\,\text{,}\,\,
\Psi([x,r]):=([x],r),\,\,\forall
[x,r]\in\mathcal{C}(S)_{\RR^+\mu}.$$ Using the commutativity of
the diagram of Figure \ref{diag1}, it is easy to see that $\Psi$
is also a symplectomporphic isometry. Namely,
$$
(\Psi\circ\pi_{1\RR^+\mu})^*(\eta_{\RR^+\mu}\wedge dr^2
+r^2d\eta_{\RR^+\mu})=i_{1\RR^+\mu}^*(\eta\wedge dr^2 +r^2 d\eta),$$ and
$$
\Psi^*(\mathcal{C}(\operatorname{g}_{\RR^+\mu}))=\mathcal{C}(\operatorname{g})_{\RR^+\mu},
$$
where
$i_{1\RR^+\mu}:\Phi_s^{-1}(\RR^+\mu)\rightarrow\mathcal{C}(S)$,
$\pi_{1\RR^+\mu}:\Phi_s^{-1}(\RR^+\mu)\rightarrow\mathcal{C}(S)_{\RR^+\mu}$,
and $\pi_{\RR^+\mu}:\Phi^{-1}(\RR^+\mu)\rightarrow(S)_{\RR^+\mu}$
are the canonical inclusion and $K_\mu$-projections,
respectivelly.
\begin{figure}
\small
$$ \xymatrix{\\ \Phi^{-1}(\RR^+\mu)\times \RR^+\ar[r]^{\hspace{6mm}
\pi_{1\RR^+\mu}}\ar[d]_{\pi_{\RR^+\mu}\times\operatorname{id}_{\RR^+}}
& \mathcal{C}(S)_{\RR^+\mu}\ \ar[r]^{\Psi}&
S_{\RR^+\mu} \times\RR^+\\
S_{\RR^+\mu} \times\RR^+ \ar @{} [ur] |{\simeq}
\ar[urr]_{\operatorname{id}_{S_{\RR^+\mu} \times\RR^+ }}& }
$$
\caption{Commutative diagram used in the proof of Lemma
\ref{lema4}}\label{diag1}\end{figure}
\end{proof}

\begin{co}
\label{red--raza--cotang}
 Let $Q$ be a differentiable manifold of
real dimesnion $n$, $G$ a finite dimensional Lie group acting
smoothly on $Q$. Denote by $\mu$ an element of the dual Lie
algebra $\mathfrak{g}^*$ and by $K_\mu$ its kernel group. Assume
that $K_\mu$ acts freely and properly on $J^{-1}(\RR^+\mu)$, with
$J:T^*Q\rightarrow\mathfrak{g}^*$ the canonical momentum map
associated to the $G$-action. Then the ray reduced space
$(T^*(Q))_{\RR^+\mu}$ is embedded by a map preserving the
symplectic structures onto a subbundle of $T^*(Q/K_\mu)$.
\end{co}

\begin{proof}
Note that the symplectic cone of the cosphere bundle of $Q$ is exactly
$T^*Q\setminus\{o_{T^*Q}\}$. Applying Theorems $3.1$ and $3.2$ in
\cite{dragulete--ornea--ratiu}, and the above lemma the conclusion of the Corollary follows.
\end{proof}

Recall that a celebrated theorem of Boothby and Wang (see Section
3.3 in \cite{blair}) states that if the contact manifold $(M, \eta)$ is
also compact and regular, then it admits a contact form whose Reeb
vector field generates a free, effective $S^1$-action on it. 
A contact structure is regular if it admits a regular Reeb vector field $\mathfrak{R}$, i.e.
any point in $M$ has a cubical neighborhood such that all the integral curves of $\mathfrak{R}$
pass at most once through this neighborhood.
Even
more, $M$ is the bundle space of a principal circle bundle
$\pi:M\rightarrow N$ over a symplectic manifold of dimension $2n$
with symplectic form $\omega$ determining an integer cocycle. In
this case, $\eta$ is a connection form on the
bundle $\pi:M\rightarrow N$ with curvature form
$d\eta=\pi^*\omega$. $N$ is actually the space of leaves of the
characteristic foliation on $M$ (i.e. the $1$-dimesional foliation
defined by the Reeb vector field of $\eta$). If $M=S$ is a
Sasakian manifold, then $N$ becomes a Hodge manifold and the
fibers of $\pi$ are totally geodesic. This case was treated by Y.
Hatakeyama in \cite{hatakeyama}. Even more, in
\cite{boyer-galicki1}, Theorem 2.4 it was proved that $S$ is
Sasakian-Einstein if and only if $N$ is K\"ahler-Einstein with
scalar curvature $4n(n+1)$ and that all the above still holds in
the category of orbifolds if $S$ is quasi-regular, i.e. all the
leaves of the characteristic foliation are compact.
\begin{pr}
\label{boothbywang} Let $\pi:(S,\operatorname{g})\rightarrow
(N,\operatorname{h})$ be the Boothby-Wang fibration associated
to the quasi-regular, compact, Sasakian manifold $S$. Suppose a
connected Lie group $G$ acts by strong contactomorphisms on
$(S,\operatorname{g})$ with momentum map
$J_S:S\rightarrow{\mathfrak{g}}^*$. Let $\mu$ be an element of
$\mathfrak{g}^*$, with kernel group $K_\mu$. Assume that the
action of $K_\mu$ on $J^{-1}(\RR^+\mu)$ is proper and by
isometries and that $\operatorname{ker}
\mu+\mathfrak{g}_\mu=\mathfrak{g}$. Then, the reduced space of $N$
at $\mu$ is well defined and there is a canonical Boothby-Wang
fibration of the reduced spaces:
$$
\tilde\pi: S_{\RR^+\mu}\rightarrow N_{\RR^+\mu}.
$$

\end{pr}
\begin{proof}
Denote by $\eta$ the contact form of the Boothby-Wang fibration
and by $\mathfrak{R}$ its Reeb vector field. Since $[\mathfrak{R},
\xi_S]=0$ for any $\xi\in\mathfrak{g}$ and $G$ is connected, the
action generated by the Reeb vector field commutes with the action
of $G$. Hence there is a well defined action of $G$ on $N$. Even
more, this action is by symplectomorphisms. If
$J_S:S\rightarrow\mathfrak{g}^*$ is the equivariant momentum map
associated to the $G$-action on $S$, the induced application
$$
J_N:N\rightarrow\mathfrak{g}^* \,\text{,}\, J_N(\pi(x)):=J_S(x),
$$
is well defined for any $x\in S$. Indeed, if $\Phi^t_{\mathfrak{R}}$ is the flow
of the Reeb vector field, we have
\begin{align*}
J_S(\Phi^t_{\mathfrak{R}}(x))(\xi)=\eta(\Phi^t_{\mathfrak{R}}(x))(\xi_S(\Phi^t_{\mathfrak{R}}(x))&
=((\Phi^t_{\mathfrak{R}})^*\eta)(x)(\xi_S(x))
=\eta(x)(\xi_S(x))\\
 =J_S(x)(\xi),
\end{align*}
for any $\xi\in\mathfrak{g}$ and any $x\in S$. This proves that
$J_N$ is well defined. Using the fact that $\pi^*\omega=d\eta$, it
is easy to see that $J_N$ is an equivariant momentum map
associated to the $G$-action on $N$. We also have that
$\pi(J_S^{-1}(\RR^+\mu))=J_N^{-1}(\RR^+\mu)$ and obviously the
action of $K_\mu$ on $J_N^{-1}(\RR^+\mu)$ is proper and by
isometries. Therefore, the quotient space $N_{\RR^+\mu}$ is a well
defined symplectic orbifold and the induced projection
$\tilde{\pi}:S_{\RR^+\mu}\rightarrow N_{\RR^+\mu}$ becomes a
Boothby-Wang fibration.
\end{proof}

\section{Conformal Hamiltonian Vector Fields}

In this section we will study the dynamical behavior of conformal
Hamiltonian systems. This class of systems comprises mechanical,
non autonomous systems with friction or Rayleigh dissipation. The
definition of conformal Hamiltonian vector fields appeared for the
first time in the work of McLachlan and Perlmutter, see
\cite{mclachlan--perlmutter}. In this section we will see that in
the presence of symmetries the solutions of conformal Hamiltonian
systems preserve the ray pre-images of the momentum map, but not
the point pre-images used in the construction of the
Marsden-Weinstein quotient. Therefore, the right tool for the
study of symmetries of these systems is the ray reduction and not
the point one. We will also enlarge the class of conformal
Hamiltonian systems previously defined and we will complete their
Lie-Poisson reduction with the general ray
reduction.

Recall that the energy of autonomous Hamiltonian systems is conserved. 
If they are endowed with an appropriate symmetry group $G$, then they also
obey an other conservation law. Namely, if
$H\in\mathcal{C}^\infty(M )$ is the $G$-invariant Hamiltonian,
$J:M\rightarrow\mathfrak{g}^*$ an associated equivariant momentum
map, the pre-images $\{J^{-1}(\mu)|\mu\in\mathfrak{g}^*\}$ are
invariant submanifolds of the Hamiltonian vector field. In
symplectic geometry this conservation property is known as the
\emph{Noether theorem} and it states that if $t\rightarrow c(t)$
is a solution of the Hamiltonian system starting at the point
$x_0$ with momentum $J(x_0)=\mu$, then at any time $t$ the
solution will have the same momentum $\mu$. In other words, the
Hamiltonian flow leaves the connected components of $J^{-1}(\mu)$
invariant and commutes with the group action. Hence, it projects
on $M_\mu$ onto another Hamiltonian flow corresponding to the
smooth function $H_\mu\in\mathcal{C}^\infty(M_\mu)$ defined by
$H_\mu\circ\pi_\mu=H\circ i_\mu$. The triple $(M_\mu,\omega_\mu,
X_{H_\mu})$ is called the \emph{reduced Hamiltonian system}. Of
course, in this setup appropriate symmetries refer to a proper,
free action which ensures the smoothness of the quotient $M_\mu$.
This is a classical result of J. Marsden and A. Weinstein. For the
proof and physical examples, see \cite{marsden--weinstein--unu}
and \cite{marsden--weinstein--doi}.

However, in physics there are a lot of simple mechanical systems
whose energy is not conserved, but dissipated. One class of such systems
is the class of conformal Hamiltonians. In 
\cite{mclachlan--perlmutter} and in the following paragraph we will briefly recall the definition and some 
of their properties. After, we will show how to extend the class of conformal Hamiltonian
systems.

In this sction $(M,\omega=-d\theta)$ will be an exact symplectic manifold.
The vector field $X^k_H$ on $M$ is \emph{conformal with real
parameter $k$} if $i_{X^k_H}\omega=dH-k\theta$ for a smooth Hamiltonian 
$H$. This condition is equivalent to
$L_{X^k_H}=-k\omega$. Note that the hypothesis of exactness of the symplectic form does not
restrain the generality since a symplectic manifold admits a
vector field $X^k_H$ with $L_{X^k_H}=-k\omega$ if and only if it
is exact. If, in addition, $H^1(M)=0$, then all the conformal
vector fields on $M$ are given by
$$
\{X_H+kZ|H\in\mathcal{C}^\infty(M)\},
$$
where $Z$ is the Liouville vector field defined by $i_Z\omega=-\theta$. For the proof, see
Proposition $1$ in \cite{mclachlan--perlmutter}. It was noticed by
the authors of this article that, in the case of Lie group symmetries, the conformal
Hamiltonian vector fields have a special behaviour with respect to
the associated momentum map. Namely,

\begin{pr}
\label{noether} 
Let $G$ be a Lie group which acts on
$(M,\omega=-d\theta)$ leaving the $1$-form $\theta$ invariant and
$H$ a smooth, $G$-invariant function on $M$. Denote by
$J:M\rightarrow\mathfrak{g}^*$ the associated $G$-equivariant
momentum map. Then, $X^k_H$ is a $G$-invariant vector field for
any real $k$ and its flows preserves the ray pre-images of the
associated momentum map as follows:
$$
J(x(t))=e^{-kt}J(x(0)),
$$
for any integral curve $x$ of $X^k_H$ and any time $t$.
\end{pr}

In other words, the motion is constrained to a ray of momentum
values entirely determined by the initial momentum. Hence, the ray
pre-images of the momentum map are invariant submanifolds for the
conformal Hamiltonian vector fields. In the hypothesis of Proposition 
\ref{noether}, with $M$ the
cotangent bundle of a Lie group $G$, the authors have performed
the conformal Lie Poisson reduction and reconstruction of
solutions for conformal Hamiltonian vector fields. However, they
could not exploit the ray momentum conservation, nor perform a
reduction which uses not only the group invariance, but also the
ray-momentum one. Proposition \ref{noether} and Theorem
\ref{reducere--unu}, immediately suggest that the appropriate
method of reduction for conformal Hamiltonian vector fields is the
ray reduction constructed in Section \ref{sectiunea--trei}.

But before passing to details, we want to show how to generalize
the definition of conformal Hamiltonian vector fields in order to
include in this study more physical systems. Let us first recall the example of
Rayleigh systems. On the canonical symplectic manifold 
$(\RR^{2n},q,p, \omega=dq\wedge dp)$ they are defined by

\begin{equation}
\label{rayleigh--sist} \left\{ \begin{array}{ll}
\dot{q} & = \frac{\partial{H}}{\partial{p}}  \\
 \dot{p} &=-\frac{\partial{H}}{\partial{q}}-R(q)\frac{\partial{H}}{\partial{p}},
\end{array}\right.
\end{equation}

where 
$H=T+V(q)\,\text{,}\,T=\frac{1}{2}p^T M(q)p$, $M$
positive definite. If $R$ is positive, they dissipate energy since $dH=-R(q)\langle\frac{\partial{H}}{\partial{p}}, \frac{\partial{H}}{\partial{p}}\rangle$. Of course the system \eqref{rayleigh--sist} is conformal Hamiltonian with parameter $k$ if and only if $R(q)=kM(q)^{-1}$. Howevere, if the real parameter $k$ is replaced by the real function $f(q,p)$, then the vector field defining \eqref{rayleigh--sist} is characterized by the
equality $i_X\omega=dH-f\theta$.

These examples suggest the following enlarged definition of a
conformal Hamiltonian vector field on an exact symplectic
manifold.

\begin{de}
The vector field $X^f_H$ on the symplectic manifold $(M,\omega=-d\theta)$
is \emph{conformal Hamiltonian} with conformal parameter the
smooth function $f$ and smooth Hamiltonian $H$ if
$i_{X^f_H}\omega=dH-f\theta$.
\end{de}

\begin{re}
Observe that if $H^1(M)=\{0\}$, $X^f_H$ is conformal Hamiltonian if
and only if $L_{X^f_H}\omega=-d(f\theta)$.
\end{re}

\begin{re}
The conformal Hamiltonian $X^f_H=X_H+Z_f$ is the summ of the
Hamiltonian vector field determined by $H$ and the vector field 
uniquely determined by the relation $i_Z\omega=-f\theta$. In local
coordinates $(q,p)$, $Z$ is given by
$fp\frac{\partial}{\partial{p}}$.
\end{re}

The next proposition shows that this enlarged class of conformal
Hamiltonians behaves well in the presence of symmetries.

\begin{pr}
\label{noether--doi} Let $G$ be a Lie group which acts on
$(M,\omega=-d\theta)$ leaving the $1$-form $\theta$ invariant, $H$
and $f$ smooth, $G$-invariant functions on $M$. Denote by
$J:M\rightarrow\mathfrak{g}^*$ the associated $G$-equivariant
momentum map. Then, $X^f_H$ is a $G$-invariant vector field and
its flow preserves the ray pre-images of the associated momentum
map as follows:
$$
J(x(t))=e^{e^{\int_0^t -f(x(s))\mathrm{d}s}}J(x(0)),
$$
for any integral curve $x$ of $X^f_H$ and any time $t$.
\end{pr}

\begin{proof}
Denote by $\phi$ the action of $G$ on $M$. Then, for any $g\in G$
we have
\begin{equation}
\label{unu}
\phi_g^*(i_{X^f_H}\omega)=\phi^*_g(dH-f\theta)=dH-f\theta=i_{X^f_H}\omega,
\end{equation}
since $f$ and $H$ are $G$-invariant. On the other hand,
\begin{equation}
\label{doi} \phi_g^*(i_{X^f_H}\omega)=i_{\phi^*_g
X^f_H}\phi^*_g\omega=i_{\phi^*_g X^f_H}\omega.
\end{equation}
Since $\omega$ is non-degenerate, \eqref{unu} and \eqref{doi}
imply that $X^f_H$ is $G$-invariant.

First recall that any exact symplectic manifold admits an
equivariant momentum map given by
$J:(M,\omega=d\theta)\rightarrow\mathfrak{g}^*$, $\langle
J(x),\xi\rangle:=\theta(\xi_M)(x)$, for any $x\in M$ and
$\xi\in\mathfrak{g}$. Now, let $x(t)$ be an integral curve of
$X^f_H$. Then,
\begin{align*}
\frac{d}{dt}\langle J(x(t)),\xi\rangle=TJ^\xi (X^f_H(x(t)))=\omega
(x(t))(X^f_H(x(t)), \xi_M(x(t)))=\\dH(\xi_M(x(t))-f(x(t))\theta
(\xi_M(x(t)))=-f(x(t))J^\xi(x(t)).
\end{align*}
Hence, $J^\xi(x(t))=e^{e^{\int_0^t
-f(x(s))\mathrm{d}s}}J^\xi(x(0))$ for any $\xi\in\mathfrak{g}$ and
any time $t$.
\end{proof}

\begin{re}
Note that if $\theta$, $f$, and $H$ are $K_\mu$-invariant, with
$K_\mu$ the kernel group associated to $\mu\in\mathfrak{g}^*$,
then the corresponding conformal Hamiltonian is also
$K_\mu$-invariant. Even more, if the $K_\mu$-action is proper and
free $X^f_H$ projects onto a conformal Hamiltonian with parameter
function and Hamiltonian canonically induced by $f$ and $H$.
\end{re}

\begin{de}
If in the hypothesis of the above remark one replaces $K_\mu$ with
$G$, the point $x\in M$ is called a \emph{relative equilibrium}
(or \emph{relative periodic}) point of $X^f_H$ if it descends
through the projection $M\mapsto M/G$ onto an equilibrium (or
periodic) point of the reduced conformal Hamiltonian.
\end{de}

Proposition \ref{noether--doi} suggests that the ray reduction is
a natural tool for the study of conformal Hamiltonian systems.
Indeed,

\begin{pr}
\label{conform} Consider $(M,\omega=-d\theta)$ an exact symplectic
manifold endowed with the smooth action of a Lie group $G$. Choose
an element $\mu$ in $\mathfrak{g}^*$ with kernel group $K_\mu$.
Denote by $J:M\rightarrow\mathfrak{g}^*$ the associated
equivariant momentum map defined by $J(x)(\xi):=i_{\xi_M}\theta$,
for any $x\in M$ and $\xi\in\mathfrak{g}$ with infinitesimal
isometry $\xi_M$. Suppose that all the hypothesis of Theorem
\ref{reducere--unu} are fulfilled and $X^f_H$ is a conformal
Hamiltonian vector field with $H$ and $f$ $K_\mu$-invariant
functions. Then,
\begin{itemize}
\item the flow of $X^f_H$ induces a flow on the ray reduced space
$M_{\RR^+\mu}$ defined by
$$
\pi_{\RR^+\mu}\circ\Phi_t\circ
i_{\RR^+\mu}=\Phi_t^{\RR^+\mu}\circ\pi_{\RR^+\mu}.
$$
\item the vector field generated by the flow $\Phi^{\RR^+\mu}_t$
is conformal Hamiltonian $(X^f_H)_{\RR^+\mu}$ with
$$
f_{\RR^+\mu}\circ\pi_{\RR^+\mu}=f\circ
i_{\RR^+\mu}\,\text{,}\,H_{\RR^+\mu}\circ\pi_{\RR^+\mu}=H\circ
i_{\RR^+\mu}.
$$ The vector fields $X^f_H$ and $(X^f_H)_{\RR^+\mu}$
are $\pi_{\RR^+\mu}$-related.
 \item a point $x\in M$ with momentum $\mu$ is a relative equilibrium of
$X^f_H$ if and only if there is an element $\xi$ of the ray
isotropy algebra $\mathfrak{g}_{\RR^+\mu}$ such that
$X^f_H(x)=\xi_M(x)$ or, equivalently, $\Phi_t(x)=\exp t\xi\cdot
x$, for any time $t$. The relative equilibria of $X^f_H$ with
momentum $\mu$ coincide via the $\pi_{\RR^+\mu}$-projection with
the equilibria of $(X^f_H)_{\RR^+\mu}$, or, equivalently, with the
points $x\in M$ with momentum $\mu$ for which there is a
$\xi\in\mathfrak{g}_{\RR^+\mu} $ such that
\begin{equation}
\label{echilibru--relativ} d(J^{\xi}-H)(x)=f(x)\theta (x).
\end{equation}
\item a point $x\in M$ with momentum $\mu$ is a relative periodic
point of $X^f_H$ if and only if there is an element $g$ of the
kernel group $K_\mu$ and a positive constant $\tau$ such that
$\Phi_{t+\tau}(x)=g\Phi_t(x)$ at any time $t$.
\end{itemize}
\end{pr}

\begin{re}
Note that, in local symplectic coordinates $(q,p)$, condition
\eqref{echilibru--relativ} is equivalent to
\begin{equation}
\label{echil--relativ}
 \left\{ \begin{array}{ll}
 pf &=\frac{\partial{(J^\xi-H)}}{\partial{q}}  \\
0 &=-\frac{\partial{(J^\xi-H)}}{\partial{p}}.
\end{array}\right.
\end{equation}
\end{re}

\begin{proof}
The first two points of the theorem are a direct consequence of
Proposition \ref{noether--doi}. For the rest, suffice it to use
the definition of a conformal Hamiltonian vector field, the
relation $\omega(\xi_M,\cdot)=dJ^\xi (\cdot)$, and Proposition
\ref{noether--doi}.
\end{proof}

\begin{ex}
\textbf{The reduction of a Rayleigh system on
$T^*(\RR^{2*}\times\RR^{2*})$.} 

On
$(T^*(\RR^{2*}\times\RR^{2*}),(q,p))\simeq
\left((\RR^{2*}\times\RR^{2*})\times\RR^4,(q_1,q_2, p_1, p_2)\right)$ consider
the Rayleigh system given by $H(q,p)=\frac{1}{2}(\Vert
q\Vert^2+\Vert p\Vert^2)$ and $f(q,p)=\Vert q_1\Vert^2+\Vert p_1\Vert^2$. Consider the
cotangent lift of the rotation action of $S^1\times S^1$ on
$\RR^{2*}\times\RR^{2*}$. The reason for restricting $\RR^4$ to
$(\RR^{2*}\times\RR^{2*})$ is to have free symmetries. The action is also
proper and $H$ and $f$ are $S^1\times
S^1$-invariant. Let $\mu:=\langle (0,1),\cdot\rangle$ be an
element of $(\RR\times\RR)^*$, the dual of the Lie algebra of
$S^1\times S^1$. Then $K_\mu=\{e\}\times S^1$ and
$\mathfrak{k}_\mu=\{0\}\times\RR$. The momentum map associated to
the $S^1\times S^1$-action is given by
$$
J:\RR^{2*}\times\RR^{2*}\times\RR^4\rightarrow
(\RR\times\RR)^*\,\text{,}\, J(q,p)=(q_1\cdot
\bar{p}_1^{T},q_2\cdot \bar{p}_2^{T}),
$$
for any $(q,p)=(q_1,q_2,p_1,p_2)\in\RR^4\setminus\{0\}\times\RR^4$
with $\bar{p}_i^T=(p_{i2},-p_{i1})$, $i=1,2$ and
$J^{-1}(\RR^+\mu)=\{(q,p)\in
(\RR^4\setminus\{0\})\times\RR^4|q_1\cdot
\bar{p}_1^{T}=0\,\text{,}\,q_2\cdot \bar{p}_2^{T}\in\RR^+ \}$. By
Theorem \ref{red--raza--cotang}, the ray reduced space
$\left(T^*(\RR^{2*}\times\RR^{2*})\right)_{\RR^+\mu}$ is embedded
in
 $T^*(\frac{\RR^{2*}\times\RR^{2*}}{\{e\}\times S^1})\newline\simeq
T^*(\RR^2\setminus\{0\}\times(0,\infty))$. The reduced Rayleigh
system is given by 
$$
H_{\RR^+\mu}(q_1,s_1,p_1, s_2)=\frac{1}{2}(\Vert
q_1\Vert^2+\Vert
p_1\Vert^2+\Vert s\Vert^2),\,\,R_{\RR^+\mu}(q)=\Vert q_1\Vert^2+\Vert p_1\Vert^2.
$$
One can easily check that the only relative equilibrium points are given by $q_1=(0,0), q_2=(0,\alpha), p_1=(0,0), p_2=(-\alpha,0)$ with corresponding velocity $\xi=(0,1)\in\mathfrak{g}_{\RR^+\mu}=\RR\times\RR$.

\end{ex}


\section{Ray Reductions of Cotangent and Cosphere Bundles of a Lie
Group}
\label{Sectiunea--sase}

In this section we will determine the ray reduced spaces for
lifted actions on cotangent and cosphere bundles. We will show
that these ray reduced spaces are universal in the sense that any
(symplectic) contact (ray) reduced space can be recovered from the
(ray) reduced space of a (cotangent) or cosphere bundle.

Let $G$ denote a $d$-dimensional Lie group with Lie algebra
$\mathfrak{g}$. $G$ acts on itself by left translations. This action
lifts canonically to an action on $T^*G$ which admits an equivariant
and right invariant momentum map

\begin{equation}
\label{apl--moment--st} J_L:T^*G\rightarrow
\mathfrak{g}^*\quad\text{,}\quad J_L(\alpha_g):=T^*_e R_g
(\alpha_g).
\end{equation}

Similarly, for right translations we can construct the equivariant
and left invariant momentum map

\begin{equation}
\label{apl--moment--dr} J_R:T^*G\rightarrow
\mathfrak{g}^*\quad\text{,}\quad J_R(\alpha_g):=T^*_e L_g
(\alpha_g).
\end{equation}

 Denote by $\mathcal{O}_\mu$ the coadjoint orbit of an
element  $\mu$ of $\mathfrak{g}^*$ and let $\mathcal{O}_{\RR_+\mu}$
be its cone defined by \eqref{conul--orbitei--coadj}.

Since for the ray-reduction the role of the coadjoint orbit will
be played by a diagonal product of its cone and the
quotient of $G$ by the corresponding kernel group, we will now
describe their manifold structure. We will see that, in general,
$\mathcal{O}_{\RR_+\mu}$ is an immersed smooth submanifold of
$\mathfrak{g}^*$. 
\begin{de}
Let $G_{\RR^+\mu}$ be the \textit{ray isotropy group} of $\mu$
defined by $G_{\RR^+\mu}:=\{g\in G\,|\,Ad^*_g\mu=r_g\mu,\, \text{for
a}\, r_g\in\RR^+\}$.
\end{de}

\begin{lm}
The ray isotropy group $G_{\RR^+\mu}$ is a closed Lie subgroup of
$G$. Its Lie algebra is given by
$$
\mathfrak{g}_{\RR^+\mu}=\{\xi\in\mathfrak{g}\,|\,ad^*_\xi\mu=r_\xi\mu\quad\text{for
a}\quad r_\xi\in\RR\}.
$$
\end{lm}

\begin{proof}
We have the following sequence of subgroups $G_\mu<G_{\RR^+\mu}<G$.
To prove that the ray isotropy group is closed in $G$, suppose
$(g_n)_{n\in\NN}$ is a convergent sequence in $G_{\RR^+\mu}$ with
$\lim\limits_{n\to\infty}g_n=g\in G$. Then
$\lim\limits_{n\to\infty}Ad^*_{g_n}\mu=(\lim\limits_{n\to\infty}r_{g_n})\mu=Ad^*_g
\mu$, for $(r_{g_n})_{n\in\NN}$ a convergent sequence of positive
numbers. Since the coadjoint map is linear and $\mu\neq 0$,
$\lim\limits_{n\to\infty}r_{g_n}$ is a strictly positive number and
hence $g\in G_{\RR^+\mu}$. Thus, the ray isotropy group is closed.
To determine its Lie algebra, let first $\xi$ be an element of
$\mathfrak{g}_{\RR^+\mu}$. We want to show that $\exp (t\xi)$
belongs to $G_{\RR^+\mu}$ for arbitrary $t\in\RR$. Then
$$
\frac{d}{dt}Ad^*_{\exp t\xi}\mu=Ad^*_{\exp t\xi}(ad^*_\xi
\mu)=Ad^*_{\exp t\xi}(r_\xi \mu)=r_\xi Ad^*_{\exp t\xi}\mu.
$$
We have used the following formula
\begin{equation}
\label{derivare--curbe--coadjuncta}
\frac{d}{dt}Ad^*_{g(t)}\mu(t)=Ad^*_{g(t)}\left(ad^*_{\xi(t)}\mu(t)
+\frac{d\mu}{dt}\right),
\end{equation}
where $\xi(t)=T_{g(t)} R^{-1}_{g(t)}(\frac{dg}{dt})$ and $g(t)$,
$\mu(t)$ are smooth curves in $G$ and $\mathfrak{g}^*$,
respectively. It follows that $Ad^*_{\exp t\xi}\mu=e^{r_\xi t}\mu$
and $\exp t\xi\in G_{\RR^+\mu}$, for every real $t$. For the reverse
inclusion, suppose $\xi$ is an element of the Lie algebra of the ray
isotropy group. Then we know that $\exp t\xi\in G_{\RR^+\mu}$ and
$Ad^*_{\exp t\xi}\mu=r_t \mu$ with $r_t$ a positive real number for
every $t\in\RR$. Deriving at zero the above equality, we obtain that
$ad^*_\xi \mu=\left(\left.\frac{d}{dt}\right|_{t=0} r_t\right)\mu$,
completing thus the proof of this lemma.
\end{proof}

\begin{re}
As we saw in the proof of Theorem \ref{reducere--unu}, if the Lie
group $G$ acts in a Hamiltonian way on the manifold $M$ and this
action admits an equivariant momentum map
$J:M\rightarrow\mathfrak{g}^*$, then for every $x\in
J^{-1}(\RR^+\mu)$ we have that
$$
J^{-1}(\RR^+\mu)\cap (G\cdot x)=G_{\RR^+\mu}\cdot
x\quad\text{,}\quad T_x(G_{\RR^+\mu}\cdot x)=T_x(G\cdot x)\cap
T_x(J^{-1}(\RR^+\mu)).
$$
\end{re}

\begin{lm}
\label{bijectie}
 The ray isotropy group $G_{\RR^+\mu}$ acts on
$G\times\RR^+$ by $ g'\cdot (g,r)\rightarrow (g'g,
\frac{r}{r_{g'}})$, where $Ad^*_{g'}\mu=r_{g'}\mu$. This action is
free and proper and, therefore, the twisted product
$G\times_{G_{\RR^+\mu}}\RR^+$ is well defined. Even more, the
surjective map
$$
f:G\times\RR^+\rightarrow\mathcal{O}_{\RR^+\mu}\quad\text{,}\quad
f(g,r):=Ad^*_{g}(r\mu)
$$
descends to a diffeomorphism on the twisted product
$G\times_{G_{\RR^+\mu}}\RR^+$.
\end{lm}
\begin{proof}
Since it consists of direct calculations, we skip the proof of this
Lemma.
\end{proof}

\begin{re}
Note that the above Lemma implies that the dimension of the cone coadjoint orbit
at $\mu$ is given by $\dim \mathcal{O}_{\RR^+\mu}=\dim
G+1-\dim G_{\RR^+\mu}$.
\end{re}

For technical reasons we need a precise description of the tangent
space of the cone coadjoint orbit.

\begin{lm}
\label{sptanglaorbita}
 Let $\mathcal{O}_{\RR^+\mu}$ be the cone of the coadjoint
orbit through $\mu\in\mathfrak{g}^*$. Then its tangent space at $\mu$
is given by
$$
T_\mu
\mathcal{O}_{\RR^+\mu}=\{ad^*_\xi\mu+r\mu\,|\,r\in\RR\,\text{,}\,\xi\in\mathfrak{g}\}.
$$
\end{lm}

\begin{proof}
Consider the smooth curve in $\mathcal{O}_{\RR^+\mu}$ given by $\mu
(t):=Ad^*_{\exp (t\xi)}(e^{tr}\mu)$, where $r$ is an arbitrary real
number. Note that $\mu (0)=\mu$ and
$\left.\frac{d}{dt}\right|_{t=0}\mu
(t)=\xi_{\mathfrak{g}^*}(\mu)+r\mu=ad^*_\xi \mu+r\mu$. Therefore,
$A:=\{ad^*_\xi\mu+r\mu\,|\,r\in\RR\,\text{,}\,\xi\in\mathfrak{g}\}\subset
T_\mu \mathcal{O}_{\RR^+\mu}$.

Let
$\mathfrak{g}=\mathfrak{g}_{\RR^+\mu}\oplus\mathfrak{m}_{\RR^+\mu}$
be a splitting of $\mathfrak{g}$, and $\{\xi_1,\cdots\xi_k\}$,
$\{\xi_{k+1},\cdots\xi_d\}$ basis of $\mathfrak{g}_{\RR^+\mu}$ and
$\mathfrak{m}_{\RR^+\mu}$, respectively. It is easy to see that the
set $\{\xi_{k+1\,\mathfrak{g}^*}(\mu), \cdots,
\xi_{d\,\mathfrak{g}^*}(\mu), \mu\}$ forms a basis of $A$. And since
$\dim (\{\xi_{k+1\,\mathfrak{g}^*}(\mu), \cdots,
\xi_{d\,\mathfrak{g}^*}(\mu), \mu\})$ is $d+1-\dim (G_{\RR^+ \mu})$,
it follows that $A= T_\mu \mathcal{O}_{\RR^+\mu}$.
\end{proof}

\begin{pr}
\label{subvarietate}
 The cone coadjoint orbit $\mathcal{O}_{\RR^+\mu}$ is an
initial  Poisson submanifold of $\mathfrak{g}^*$ and if the
coadjoint action is proper, it is even a closed embedded
submanifold.
\end{pr}

\begin{proof}
The simplest way to see this is to notice that the ray coadjoint
orbit of $\mu$ is actually the orbit through $\mu$ of the
following action of $G\times\RR^+$ on $\mathfrak{g}^*$
$$
(g,r)\cdot\mu':=Ad^*_{g^{-1}}r\mu',
$$
for any $(g,r)\in G\times\RR^+$ and any $\mu'\in\mathfrak{g}^*$.
Therefore, as any orbit it is an initial Poisson submanifold.
Since the coadjoint action of $G$ is proper, so is the action of
$G\times\RR^+$. Therefore, $\mathcal{O}_{\RR^+\mu}$ is a closed
embedded submanifold of $\mathfrak{g}^*$. Of course, one can
easily verify that the smooth structure of
$\mathcal{O}_{\RR^+\mu}$ as orbit of the $(G\times\RR^+)$-action
coincides with the one described in Lemma \ref{bijectie}. Indeed,
$G\times_{G_{\RR^+\mu}}\RR^+$ and
$\frac{G\times\RR^+}{(G\times\RR^+)_\mu}$ are diffeomorphic
manifolds.
\end{proof}

Since the coadjoint action of $G$ restricts to
$\mathcal{O}_{\RR^+\mu}$, we have the following.

Fix $\mu$ an element of $\mathfrak{g}^*$. Notice that the Lie
algebra of the kernel group of $\mu$, $\mathfrak{k}_\mu$ is closed
in $\mathfrak{g}_\mu$. Let $U$ be an open neighborhood of $0$ in
$\mathfrak{g}$ such that the exponential map $\exp:U\rightarrow
\exp(U)$ is a diffeomorphism. Choose $V\subset U$ a closed
neighborhood of $0$. Then $\exp (V)\subset\exp (U)$ is a closed
neighborhood of $e$. We want to show that $\exp(V)\cap K_\mu$ is
closed in $G$. Thus, suppose $(k_n)_{n\in\NN}=(\exp\xi_n)_{n\in\NN}$
is a convergent sequence of $\exp(V)\cap K_\mu$ with
$(\xi_n)_{n\in\NN}$ a sequence in $V$. Since $\exp^{-1}k_n=\xi_n$
for every $n\in\NN$, it follows that in fact
$\xi_n\in\mathfrak{k}_\mu$. Using the continuity of the exponential
map and the fact that the kernel algebra is closed in
$\mathfrak{g}$, we have that
$\lim\limits_{n\to\infty}\exp^{-1}k_n=\lim\limits_{n\to\infty}\xi_n=\xi\in\mathfrak{k}_\mu$.
Therefore
$\lim\limits_{n\to\infty}\exp\xi_n=\exp\lim\limits_{n\to\infty}\xi_n=\exp\xi\in
K_\mu$ and $\exp(V)\cap K_\mu$ is closed in $G$. A standard result
of Lie theory (see, for instance, \cite{duistermaat--kolk}),
Corollary 1.10.7) implies that the kernel group of $\mu$ is a closed
regular Lie subgroup of $G$ and the quotient $\frac{G}{K_\mu}$ is a
smooth manifold.

Now we are ready to define the manifold which will play the role of
the cotangent orbit for the ray reduction, namely the diagonal
product of the cone coadjoint orbit and the quotient of $G$ by the
corresponding kernel group
$$
\operatorname{Diag}\left(\mathcal{O}_{\RR^+\mu}\times\frac{G}{K_\mu}\right):=
\{(Ad^*_g r\mu, \hat{g})\,|\,g\in G\,\text{and}\, r\in\RR^+\}.
$$
Recall that given two surjective submersions $\pi_1:M_1\rightarrow
E$ and $\pi_2:M_2\rightarrow E$, the diagonal of $M_1\times M_2$
over $(\pi_1, \pi_2)$, $\operatorname{Diag}\left(M_1\times
M_2\right):=\{(x_1,x_2)\in M_1\times
M_2\,|\,\pi_1(x_1)=\pi_2(x_2)\}$ is a submanifold of $M_1\times M_2$
and its tangent space is given by
\begin{align*}
T_{(x_1,x_2)}(\operatorname{Diag}(M_1\times M_2)) \simeq
&\{(v_1,v_2)\in T_{x_1}M_1\times
T_{x_2}M_2\,|\,T_{x_1}\pi_1(v_1)=T_{x_2}\pi_2(v_2)\}=\\
\operatorname{Diag}(T_{x_1}M_1\times T_{x_2}M_2).
\end{align*}

 In particular, for
$\pi_1:\mathcal{O}_{\RR^+\mu}\rightarrow \frac{G}{G_{\RR^+\mu}}$
defined by $\pi_1(Ad^*_g r\mu):=\hat{g}$ and $\pi_2$ the canonical
projection from${\frac{G}{K_\mu}}$ onto $\frac{G}{G_{\RR^+\mu}}$
we obtain that $$T_{(Ad^*_g r\mu,
\hat{g})}\operatorname{Diag}\left(\mathcal{O}_{\RR^+\mu}\times\frac{G}{K_\mu}\right)
\simeq\operatorname{Diag}\left(T_{Ad^*_g
r\mu}\mathcal{O}_{\RR^+\mu}, T_{\hat{g}}
\frac{G}{K_\mu}\right)\simeq T_{Ad^*_g
r\mu}\mathcal{O}_{\RR^+\mu}.$$ More precisely, we have that
$$
T_{(Ad^*_g r\mu,
\hat{g})}\operatorname{Diag}\left(\mathcal{O}_{\RR^+\mu}\times\frac{G}{K_\mu}\right)
=\{\left(ad^*_{\xi}(Ad^*_g r\mu)+r'Ad^*_g r\mu,
\hat{\xi}_{G}(\hat{g})\right)\,|\,\xi\in\mathfrak{g}\,\text{,}\,r'\in\RR\},
$$
for any $(Ad^*_g r\mu, \hat{g})\in\mathcal{O}_{\RR^+\mu}$. Here
$\hat{\xi}_{G}(\hat{g})$ denotes the projection on
$\frac{G}{K_\mu}$ of the infinitesimal isometry associated to
$\xi$ with respect to the action by left translations of $G$ on
itself. Let $\omega^-_{\RR^+\mu}$ be the two form on
$\operatorname{Diag}\left(\mathcal{O}_{\RR^+\mu}\times\frac{G}{K_\mu}\right)$
defined by
\begin{multline}
\label{ecuatie}
 \omega^-_{\RR^+\mu}(Ad^*_g
r\mu,\hat{g})((ad^*_{\xi_1}Ad^*_g r\mu+r_1Ad^*_g r\mu,
\hat{\xi_{1G}}(\hat{g})),(ad^*_{\xi_2}Ad^*_g r\mu+r_2Ad^*_g r\mu,
\hat{\xi_{2G}}(\hat{g})))\\=-\langle Ad^*_g r\mu,[\xi_1,
\xi_2]\rangle+r_2\langle Ad^*_g r\mu, \xi_1\rangle -r_1\langle
Ad^*_g r\mu, \xi_2\rangle,
\end{multline}
for any $(Ad^*_g
r\mu,\hat{g})\in\operatorname{Diag}\left(\mathcal{O}_{\RR^+\mu}\times\frac{G}{K_\mu}\right)$
and any tangent vectors $(ad^*_{\xi_i}Ad^*_g r\mu+r_iAd^*_g r\mu,
\hat{\xi_{iG}})_{i=1, 2}\in T_{(Ad^*_g
r\mu,\hat{g})}\operatorname{Diag}\left(\mathcal{O}_{\RR^+\mu}\times\frac{G}{K_\mu}\right).$
In fact, as we will see from Theorem \ref{teorema--sase--unu},
$\left(\operatorname{Diag}\left(\mathcal{O}_
{\RR^+\mu}\times\frac{G}{K_\mu}\right), \omega^{-}_{\mathcal{O}_
{\RR^+\mu}}\right)$ is a well defined symplectic manifold. One
could also prove this directly, but we prefer to skyp the computations and apply instead
the ray reduction.

\begin{te}
\label{teorema--sase--unu}
 Consider the cotangent lift of the action by
left translations of a Lie group $G$ on itself. For every
$\mu\in\mathfrak{g}^*$ with
$\ker{\mu}+\mathfrak{g}_\mu=\mathfrak{g}$, the ray reduced space
$(T^*(G)_{\RR^+\mu}, \omega_{\RR^+\mu})$ is well defined and
symplectomorphic to the diagonal manifold
$\left(\operatorname{Diag}\left(\mathcal{O}_{\RR^+\mu}\times\frac{G}{K_\mu}\right),
\omega^-_{\RR^+\mu}\right)$ with symplectic form
$\omega^-_{\RR^+\mu}$ defined by \eqref{ecuatie}.
\end{te}

\begin{proof}
Since the cotangent lift of left translations is a free and proper
action, if $\mu$ is an element of $\mathfrak{g}^*$ with
$\ker{\mu}+\mathfrak{g}_\mu=\mathfrak{g}$ the ray reduced space at
$\mu$, $(T^*G)_{\RR^+\mu}=\frac{J_L^{-1}(\RR^+\mu)}{K_\mu}$ is
well defined. $J_R$ is the momentum map defined by
\eqref{apl--moment--st}.

Note that $J_L^{-1}(\RR^+\mu)=\{T^*_g R_{g^{-1}}(r\mu)\,|\,g\in
G\,,\,r\in\RR^+\}$. The momentum map associated to right
translations (see \eqref{apl--moment--dr}) induces the application
$\bar{J}_R:(T^*G)_{\RR^+\mu}\rightarrow\operatorname{Diag}\left(\mathcal{O}_
{\RR^+\mu}\times\frac{G}{K_\mu}\right)$ defined by
$\bar{J}_R([\alpha_g]):= (J_R(\alpha_g),\hat{g})=(Ad^*_g
r\mu,\hat{g})$, for any $\alpha_g=T^*_g R_{g^{-1}}(r\mu)$. To see
that $\bar{J}_R$ is well defined, fix an arbitrary $k\in K_\mu$.
Then,
\begin{align*}
\bar{J}_R([k\cdot\alpha_g])=\bar{J}_R([k\cdot T^*_g&
R_{g^{-1}}r\mu])= \bar{J}_R([T^*_{kg}L_{k^{-1}}T^*_g
R_{g^{-1}}r\mu])&\\&=(T^*_e L_{kg}T^*_{kg}(R_{g^{-1}}\circ
L_{k^{-1}})(r\mu), \hat{kg})=(Ad^*_g(r\mu),\hat{g}),
\end{align*}
proving thus that $\bar{J}_R$ is indeed well-defined. Since the
kernel group of $\mu$ is a subgroup of its isotropy group,
$\bar{J}_R$ is also one to one. Surjectiveness is obvious and hence
$\bar{J}_R$ is a bijection. Its inverse is given by
$\bar{J}_{R^{-1}}:\operatorname{Diag}\left(\mathcal{O}_
{\RR^+\mu}\times\frac{G}{K_\mu}\right)\rightarrow
(T^*G)_{\RR^+\mu}$, $\bar{J}_{R^{-1}}(Ad^*_g r\mu,\hat{g})=[T^*_g
R_{g^{-1}}r\mu]$.

To prove that $\bar{J}_R$ is smooth, and hence a diffeomorphism we
will use the right invariant $1$-form $\lambda\in\Lambda^1(G)$
given by $\lambda(g)(v_g):=T^*_g R_{g^{-1}}\mu(v_g)$. The graph of
$\lambda$ defines the diffeomorphism $F':G\rightarrow
J^{-1}_L(\mu)$, $F'(g):=\lambda(g)=T^*_g R_{g^{-1}}\mu$. Consider
the map $F:G\times\RR^+\rightarrow J^{-1}_L(\RR^+\mu)$ given by
$F(g,r):=F'(g)r$, for any elements $g\in G$ and $r\in\RR^+$. It is
obviously smooth and we want to show that it descends to a
diffeomorphism $\bar{F}:\operatorname{Diag}\left(G\times
_{G_{\RR^+\mu}}\RR^+)\times\frac{G}{K_\mu}\right)\rightarrow
(T^*G)_{\RR^+\mu}$, $\bar{F}(([g,r],\hat{g})=[T^*_g
R_{g^{-1}}r\mu]$. Let us first verify that it is a well defined
map. For this, let $(k,r_k)\in G\times _{G_{\RR^+\mu}}\RR^+$ with
$\hat{g}=\hat{kg}$, so that
$([(kg,\frac{r}{r_k})],\hat{kg})=([(g,r)],\hat{g})$. The equality
of the second components implies that actually $k$ belongs to the
kernel group of $\mu$. Then we obtain
\begin{multline*}
\bar{F}([(kg,\frac{r}{r_k})],\hat{kg})=[T^*_{kg}R_{(kg)^{-1}}\frac{r}{r_k}\mu]=
[T^*_{kg}R_{g^{-1}}T^*_k L_{k^{-1}} T^*_e L_k
T^*_{k}R_{k^{-1}}\frac{r}{r_k}\mu]=\\
[T^*_{kg}R_{g^{-1}} T^*_k L_{k^{-1}} Ad^*_k\frac{r}{r_k}\mu]=
[k\cdot T^*_g R_{g^{-1}}r\mu]=[T^*_g
R_{g^{-1}}r\mu]=\bar{F}([g,r],\hat{g}).
\end{multline*}
Observe that $\bar{J}_R\circ\bar{F}$ is precisely the
diffeomrophism of Lemma \ref{bijectie}. Therefore, $\bar{J}_R$ is
also a diffeomorphism. Its inverse is given by
 $$
 \bar{J}_{R^{-1}}:\operatorname{Diag}\left(\mathcal{O}_
{\RR^+\mu}\times\frac{G}{K_\mu}\right)\rightarrow
(T^*G)_{\RR^+\mu}\quad\text{,}\quad \bar{J}_{R^{-1}}(Ad^*_g
r\mu,\hat{g})=[T^*_g R_{g^{-1}} r\mu],
$$
and we can endow $\operatorname{Diag}\left(\mathcal{O}_
{\RR^+\mu}\times\frac{G}{K_\mu}\right)$ with the symplectic form
$\omega^{-}_{\mathcal{O}_
{\RR^+\mu}}:=\bar{J}^*_{R^{-1}}\omega_{\RR^+\mu}$. In order to
give the explicit description of $\omega^{-}_{\mathcal{O}_
{\RR^+\mu}}$ fix
$(Ad^*_{g}r\mu,\hat{g})\in\operatorname{Diag}\left(\mathcal{O}_
{\RR^+\mu}\times\frac{G}{K_\mu}\right)$ and two tangent vectors
$\{v_i=(ad^*_{\xi_i}Ad^*_g r\mu+r_iAd^*_g r\mu,
\hat{\xi}_{iG}(\hat{g})\}_{i=1,2}$ in\newline
$T_{(Ad^*_{g}r\mu,\hat{g})}\operatorname{Diag}\left(\mathcal{O}_
{\RR^+\mu}\times\frac{G}{K_\mu}\right)$. It follows that
$$
\omega^{-}_{\mathcal{O}_
{\RR^+\mu}}(Ad^*_{g}r\mu,\hat{g})(v_1,v_2)=\omega_{\RR^+\mu}([T^*_g
R_{g^{-1}}r\mu])(T_{(Ad^*_{g}r\mu,\hat{g})}\bar{J}_{R^{-1}}(v_1),
T_{(Ad^*_{g}r\mu,\hat{g})}\bar{J}_{R^{-1}}(v_2)).
$$
Note that
\begin{multline}
T_{(Ad^*_{g}r\mu,\hat{g})}\bar{J}_{R^{-1}}(v_i)=\\T_{(Ad^*_{g}r\mu,\hat{g})}\bar{J}_{R^{-1}}
\left(\left.\frac{d}{dt}\right|_{t=0}\left(Ad^*_{\exp
t\xi_i}e^{tr_i}Ad^*_{g}r\mu, \widehat{(\exp t\xi_i\cdot g)}\right)\right)\\=
\left.\frac{d}{dt}\right|_{t=0}\left(\bar{J}_{R^{-1}}(Ad^*_{\exp
t\xi_i}e^{tr_i}Ad^*_{g}r\mu, \widehat{(\exp t\xi_i\cdot g)})\right)\\=
\left.\frac{d}{dt}\right|_{t=0}\left(\pi_{K_\mu}(T^*_{g\exp
t\xi_i}R_{(g\exp
t\xi_i )^{-1}}re^{tr_i}\mu)\right)\\
=T_{T^*_g
R_{g^{-1}}r\mu}\pi_{K_\mu}\left(\left.\frac{d}{dt}\right|_{t=0}e^{tr_i}(T^*_g
R_{g^{-1}}r\mu)\cdot\exp t\xi_i\right)= T_{T^*_g
R_{g^{-1}}r\mu}\pi_{K_\mu}(X^{\xi_i}(T^*_g R_{g^{-1}}r\mu)),
\nonumber
\end{multline}
where $X^{\xi_i}$ is the vector field on $T^*G$ with flow given
by\newline $\Phi^i(t,\alpha_{g'}): =T^*_{g\exp t\xi_i}R_{\exp
-t\xi_i} e^{tr_i}\alpha_{g'}$, for any $\alpha_{g'}\in T^*_{g'}G$.
Then, using the fact that
$\pi_{\RR^+\mu}^*\omega_{\RR^+\mu}=i_{\RR^+\mu}^*(-d\theta)$ and the
above calculus, we obtain
\begin{multline}
 \omega^{-}_{\mathcal{O}_
{\RR^+\mu}}(Ad^*_{g}r\mu,\hat{g})(v_1,v_2)=
\pi_{K_\mu}^*\omega_{\RR^+\mu}(T^*_g R_{g^{-1}}r\mu)(X^{\xi_1}(T^*_g
R_{g^{-1}}r\mu), X^{\xi_2}(T^*_g R_{g^{-1}}r\mu))\\= -d\theta(T^*_g
R_{g^{-1}}r\mu)(X^{\xi_1}(T^*_g R_{g^{-1}}r\mu), X^{\xi_2}(T^*_g
R_{g^{-1}}r\mu))=\\-X^{\xi_1}(\theta(X^{\xi_2})(T^*_g
R_{g^{-1}}r\mu)+X^{\xi_2}(\theta(X^{\xi_1})(T^*_g
R_{g^{-1}}r\mu)+\theta([X^{\xi_1}, X^{\xi_2}])(T^*_g
R_{g^{-1}}r\mu). \nonumber
\end{multline}
Next, we want to show that
\begin{equation}
\label{ecuatia unu} \theta(X^{\xi_i})= J^{\xi_i}_R\, \text{and}\,
X^{\xi_i}(J^{\xi_j}_R)(T^*_g R_{g^{-1}}r\mu)= \langle Ad^*_g r\mu,
[\xi_i, \xi_j]\rangle +r_i \langle Ad^*_g r\mu, \xi_j\rangle,
\end{equation}
for $i=1,2$. Indeed, for any $\alpha_g\in T^*_g G$ we have
\begin{align*}
\theta(X^{\xi_i})(\alpha_g)= \langle\alpha_g, T_{\alpha_g}\pi
(X^{\xi_i}(\alpha_g))\rangle =\langle\alpha_g, \left
.\frac{d}{dt}\right |_{t=0}\pi(e^{tr_i}\alpha_g\cdot\exp
t\xi_i)\rangle \\=\langle \alpha_g, \xi_{iG}(g)\rangle
=J^{\xi_i}_R(\alpha_g).
\end{align*}
This also implies that $X^{\xi_i}$ and $\xi_{iG}$ are $\pi$-related
vector fields. And
\begin{multline}
X^{\xi_i}(J^{\xi_j}_R)(T^*_g R_{g^{-1}}r\mu)=\left
.\frac{d}{dt}\right |_{t=0}J^{\xi_j}_R(e^{tr_i}T^*_g
R_{g^{-1}}r\mu\cdot\exp t\xi_i)=\\
 \left .\frac{d}{dt}\right
|_{t=0} T_e^* L_{g\exp t\xi_i}(T^*_{g\exp t\xi_i}R_{\exp -t\xi_i}
(e^{tr_i}T^*_g R_{g^{-1}}r\mu))(\xi_j)= \\ \left .\frac{d}{dt}\right
|_{t=0} Ad^*_{g\exp t\xi_i} (e^{tr_i}r\mu)(\xi_j)=Ad^*_g (ad^*_{Ad_g
\xi_i}r\mu+r_i r\mu)(\xi_j)=\\Ad^*_g(ad^*_{Ad_g \xi_i}r\mu+r_i
r\mu)(\xi_j)= (ad^*_{\xi_i} (Ad^*_g r\mu)+r_i Ad^*_g r\mu)(\xi_j).
\nonumber
\end{multline}
Note that in the above calculation we have again used formula
\eqref{derivare--curbe--coadjuncta}. Applying \eqref{ecuatia unu},
it follows that
\begin{eqnarray*}
 \omega^{-}_{\mathcal{O}_
{\RR^+\mu}}(Ad^*_{g}r\mu,\hat{g})(v_1,v_2) =
-X^{\xi_1}(\theta(X^{\xi_2})(T^*_g R_{g^{-1}}r\mu) +
X^{\xi_2}(\theta(X^{\xi_1})(T^*_g R_{g^{-1}}r\mu)\\  +
\theta(X^{[\xi_1, \xi_2]})(T^*_g R_{g^{-1}}r\mu) = -\langle Ad^*_g
r\mu, [\xi_1,\xi_2]\rangle - r_1\langle Ad^*_g r\mu, \xi_2\rangle
\\+ \langle Ad^*_g r\mu, [\xi_2,\xi_1]\rangle + r_2\langle Ad^*_g r\mu,
\xi_1\rangle +J_R^{[\xi_1,\xi_2]}(T^*_g R_{g^{-1}} r\mu)\\ =
-\langle Ad^*_g r\mu, [\xi_1,\xi_2]\rangle + r_2\langle Ad^*_g
r\mu, \xi_1\rangle -r_1\langle Ad^*_g r\mu, \xi_2\rangle .
\end{eqnarray*}
In particular, for $g=e$ and $r=1$ we have that
\begin{eqnarray*}
\omega^{-}_{\mathcal{O}_
{\RR^+\mu}}(\mu,\hat{e})((ad^*_{\xi_1}\mu+r_1\mu, \hat{\xi}_1),
(ad^*_{\xi_2}\mu+r_2\mu, \hat{\xi_2}))=-\langle
\mu,[\xi_1,\xi_2]\rangle +r_2\langle \mu,\xi_1\rangle
\\-r_1\langle \mu, \xi_2\rangle ,\,\text{for any}\,
\xi_{\{i=1,2\}}\in\mathfrak{g}.
\end{eqnarray*}
 The first term in the
above expression is precisely $\omega^{-}_{\mathcal{O}_
{\mu}}(\mu)(ad^*_{\xi_1}\mu, ad^*_{\xi_2}\mu)$ and hence the minus
sign in the notation of the symplectic form on
$\operatorname{Diag}\left(\mathcal{O}_
{\RR^+\mu}\times\frac{G}{K_\mu}\right)$.
\end{proof}

\begin{co}
In the hypothesis of Theorem \ref{teorema--sase--unu}, the
symplectic form $\omega^{-}_{\mathcal{O}_ {\RR^+\mu}}$ defined by
\eqref{ecuatie} is $G$-invariant with respect to the following action
$$
g_1\cdot (Ad^*_g r\mu,\hat{g}):=\left(Ad^*_{g_1^{-1}}Ad^*_g
r\mu,\widehat{gg_1^{-1}}\right),
$$
for each $g_1$ in $G$ and $(Ad^*_g r\mu,\hat{g})$ in
$\operatorname{Diag}\left(\mathcal{O}_
{\RR^+\mu}\times\frac{G}{K_\mu}\right)$.
\end{co}
\begin{proof}
Fix $g_1$ in $G$ and $x:=(Ad^*_g r\mu,\hat{g})$ in
$\operatorname{Diag}\left(\mathcal{O}_
{\RR^+\mu}\times\frac{G}{K_\mu}\right)$. Let $v_\xi$ be the tangent
vector $(ad^*_\xi(Ad^*_g r\mu)+r_\xi Ad^*_g r\mu,
\hat{\xi}_G(\hat{g}))\in T_{(Ad^*_g
r\mu,\hat{g})}\operatorname{Diag}\left(\mathcal{O}_
{\RR^+\mu}\times\frac{G}{K_\mu}\right)$. Here $\xi$ is an arbitrary
element of $\mathfrak{g}$. Then, we have
\begin{multline*}
\omega^-_{\mathcal{O}_ {\RR^+\mu}}(g_1\cdot x)(g_1\cdot v_\xi,
g_1\cdot v_\eta)= \\\omega^-_{\mathcal{O}_ {\RR^+\mu}}(g_1\cdot
x)\left(\left .\frac{d}{dt}\right |_{t=0}(Ad^*_{\exp t\xi
g_1^{-1}}e^{tr_\xi}Ad^*_g r\mu), \widehat{T_g
R_{g_1^{-1}}\xi_G}(\hat{g}), \right.\\\left. \left
.\frac{d}{dt}\right |_{t=0}(Ad^*_{\exp t\eta
g_1^{-1}}e^{tr_\eta}Ad^*_g r\mu), \widehat{T_g
R_{g_1^{-1}}\eta_G}(\hat{g})\right)=\\\omega^-_{\mathcal{O}_
{\RR^+\mu}}(g_1\cdot x)\left(\left(Ad^*_{g_1^{-1}}v_\xi,
\widehat{(Ad_{g_1}\xi)}_G(\widehat{gg_1^{-1}})\right),\left(Ad^*_{g_1^{-1}}v_\eta,\widehat{(Ad_{g_1}\eta)}_G(\widehat{gg_1^{-1}})\right)\right)=\\
\omega^-_{\mathcal{O}_ {\RR^+\mu}}(g_1\cdot
x)\left(\left(v_{Ad_{g_1}\xi},\widehat{(Ad_{g_1}\xi)}_G(\widehat{gg_1^{-1}})\right),\left(v_{Ad_{g_1}\eta},\widehat{(Ad_{g_1}\eta)}_
G(\widehat{gg_1^{-1}})\right)\right)=\\-\langle
Ad^*_{gg_1^{-1}}r\mu,[Ad_{g_1}\xi,Ad_{g_1}\eta]\rangle+r_\eta\langle
Ad^*_{gg_1^{-1}}r\mu,Ad_{g_1}\xi\rangle-r_\xi\langle
Ad^*_{gg_1^{-1}}r\mu,Ad_{g_1}\eta\rangle=\\
-\langle Ad^*_g r\mu,[\xi,\eta]\rangle+r_\eta\langle Ad^*_g
r\mu,\xi\rangle-r_\xi\langle Ad^*_g r\mu,\eta\rangle=
\omega^-_{\mathcal{O}_{\RR^+ \mu}} (x)(v_\xi,v_\eta).
\end{multline*}
Therefore, $\omega^{-}_{\mathcal{O}_ {\RR^+\mu}}$ is $G$-invariant.
\end{proof}

\begin{pr}
The symplectomorphic $G$-action on
$\left(\operatorname{Diag}\left(\mathcal{O}_
{\RR^+\mu}\times\frac{G}{K_\mu}\right), \omega^-_{\mathcal{O}_
{\RR^+\mu}}\right)$ admits an equivariant momentum map
$$
-I_{\mathcal{O}_ {\RR^+\mu}}:\operatorname{Diag}\left(\mathcal{O}_
{\RR^+\mu}\times\frac{G}{K_\mu}\right)\rightarrow\mathfrak{g}^*\,\text{,}\,I_{\mathcal{O}_
{\RR^+\mu}}(Ad^*_g r\mu,\hat{g}):=-Ad^*_g r\mu,
$$
for each $(Ad^*_g r\mu,\hat{g})$ in
$\operatorname{Diag}\left(\mathcal{O}_
{\RR^+\mu}\times\frac{G}{K_\mu}\right)$.
\end{pr}

\begin{proof}
Let $\xi$ be an element of $\mathfrak{g}$ and denote by
$I^{\xi}_{\mathcal{O}_
{\RR^+\mu}}:\operatorname{Diag}\left(\mathcal{O}_
{\RR^+\mu}\times\frac{G}{K_\mu}\right)\rightarrow\RR$ the map given
by $(Ad^*_g r\mu,\hat{g})\mapsto\langle Ad^*_g r\mu,\xi\rangle$. The
infinitesimal generator associated to $\xi$ on
$\operatorname{Diag}\left(\mathcal{O}_
{\RR^+\mu}\times\frac{G}{K_\mu}\right)$ is
\begin{align*}
\xi_{\operatorname{Diag}\left(\mathcal{O}_
{\RR^+\mu}\times\frac{G}{K_\mu}\right)}(Ad^*_g r\mu,\hat{g})=\left
.\frac{d}{dt}\right |_{t=0}\left(Ad^*_{\exp (-t\xi)}Ad^*_g
r\mu,\widehat{g\exp (-t\xi)}\right)=\\(ad^*_{-\xi}(Ad^*_g
r\mu),-\hat{\xi}_G(\hat{g})),\,\text{for any}\, (Ad^*_g
r\mu,\hat{g})\in\operatorname{Diag}\left(\mathcal{O}_
{\RR^+\mu}\times\frac{G}{K_\mu}\right).
\end{align*}
Then, for all $\left(ad^*_\eta Ad^*_gr\mu+r_\eta
Ad^*_gr\mu,\hat{\eta}_G(\hat{g})\right)\in T_{(Ad^*_g
r\mu,\hat{g})}\operatorname{Diag}\left(\mathcal{O}_
{\RR^+\mu}\times\frac{G}{K_\mu}\right)$, we obtain that
\begin{align}
\label{ecuatia--doi} i_{\xi_{\operatorname{Diag}\left(\mathcal{O}_
{\RR^+\mu}\times\frac{G}{K_\mu}\right)}}\omega^{-}_{\mathcal{O}_
{\RR^+\mu}}(Ad^*_g r\mu,\hat{g})(ad^*_\eta Ad^*_gr\mu+r_\eta
Ad^*_gr\mu,\hat{\eta}_G(\hat{g}))=\\
\omega^{-}_{\mathcal{O}_ {\RR^+\mu}}(Ad^*_g
r\mu,\hat{g})\left(\left(ad^*_{-\xi}(Ad^*_g
r\mu),-\hat{\xi}_G(\hat{g})\right),\left(ad^*_\eta Ad^*_gr\mu+r_\eta
Ad^*_gr\mu,\hat{\eta}_G(\hat{g})\right)\right)=\nonumber\\
\langle Ad^*_gr\mu, [\xi,\eta]\rangle-r_\eta\langle
Ad^*_gr\mu,\xi\rangle\nonumber.
\end{align}
On the other hand,
\begin{align}
\label{ecuatia--trei}
 T_{(Ad^*_g r\mu,\hat{g})}I^\xi_{\mathcal{O}_
{\RR^+\mu}}\left(ad^*_\eta Ad^*_gr\mu+r_\eta
Ad^*_gr\mu,\hat{\eta}_G(\hat{g})\right)=\left .\frac{d}{dt}\right
|_{t=0}\langle Ad^*_{\exp
t\eta}e^{tr_\eta}Ad^*_gr\mu,\xi\rangle\\
=-\langle Ad^*_gr\mu, [\xi,\eta]\rangle+r_\eta\langle
Ad^*_gr\mu,\xi\rangle\nonumber .
\end{align}
Equalities \eqref{ecuatia--doi} and \eqref{ecuatia--trei} imply
that $X_{-I^\xi _{\mathcal{O}_
{\RR^+\mu}}}\hspace{-6mm}=\xi_{\operatorname{Diag}\left(\mathcal{O}_
{\RR^+\mu}\times\frac{G}{K_\mu}\right)}$ for all
$\xi\in\mathfrak{g}$. Hence the proof of this proposition is
complete.
\end{proof}

Recall that the symplectic difference of two symplectic manifolds
$(M_i,\omega_i)_{i=1,2}$ is $M_1\circleddash M_2:=(M_1\times
M_2,\pi^*_1\omega_1-\pi^*_2\omega_2)$, where $(\pi_i:M_1\times
M_2\rightarrow M_i)_{i=1,2}$ are the canonical projections. If the
Lie group $G$ acts on both $M_1$ and $M_2$ such that these actions
admit equivariant momentum maps
$(J_i:M_i\rightarrow\mathfrak{g}^*)_{i=1,2}$, then the diagonal
action of $G$ on the symplectic difference $M_1\circleddash M_2$
admits an equivariant momentum map given by
$J_d:=J_1\circ\pi_1-J_2\circ\pi_2:M_1\circleddash
M_2\rightarrow\mathfrak{g}^*$.

The following theorem illustrates the theoretical importance of the
diagonal product $\operatorname{Diag}\left(\mathcal{O}_
{\RR^+\mu}\times\frac{G}{K_\mu}\right)$ in the reduction procedure.
Namely, any ray reduced space can be seen as the symplectic
difference  of the initial manifold and the diagonal product of the
associated ray coadjoint orbit with the quotient of $G$ by the
kernel group.

\begin{te}[Shifting Theorem]
Let the Lie group $G$ act smoothly on the symplectic manifold
$(M,\omega)$ such that it admits an equivariant momentum map
$J:M\rightarrow\mathfrak{g}^*$. Fix $\mu$ an element of the dual Lie
algebra of $G$ and suppose that the hypothesis of Theorem
\ref{reducere--unu} are fulfilled. Then $G$ acts diagonaly on
$M\circleddash\operatorname{Diag}\left(\mathcal{O}_
{\RR^+\mu}\times\frac{G}{K_\mu}\right)$ and its symplectic reduced
space at zero is well defined. Even more,\newline
$\left(M\circleddash\operatorname{Diag}\left(\mathcal{O}_
{\RR^+\mu}\times\frac{G}{K_\mu}\right)\right)_0$ is symplectomorphic
to $M_{\RR^+\mu}$, the ray reduced space at $\mu$ of $M$.
\end{te}
\begin{proof}
The symplectic difference
$M\circleddash\operatorname{Diag}\left(\mathcal{O}_
{\RR^+\mu}\times\frac{G}{K_\mu}\right)$ has symplectic form
$\pi^*_1\omega-\pi^*_2\omega^-_{\mathcal{O}_{\RR^+\mu}}$ and
momentum map
$J_d:=J\circ\pi_1+I_{\mathcal{O}_{\RR^+\mu}}\circ\pi_2$. Of course,
$\pi_1:M\circleddash\operatorname{Diag}\left(\mathcal{O}_
{\RR^+\mu}\times\frac{G}{K_\mu}\right)\rightarrow M$ and
$\pi_2:M\circleddash \operatorname{Diag}\left(\mathcal{O}_
{\RR^+\mu}\times\frac{G}{K_\mu}\right)\rightarrow\operatorname{Diag}\left(\mathcal{O}_
{\RR^+\mu}\times\frac{G}{K_\mu}\right)$ are the canonical
projections. It is easy to check that in the hypothesis of Theorem
\ref{reducere--unu}, the $0$-symplectic reduced space is well
defined.

Let $\phi:J^{-1}(\RR^+\mu)\rightarrow M\circleddash
\operatorname{Diag}\left(\mathcal{O}_
{\RR^+\mu}\times\frac{G}{K_\mu}\right)$ be the map defined by $x\in
J^{-1}(\RR^+\mu)\mapsto (x,(-J(x),\hat{e}))$. Denote by $[\phi]$ its
$(K_\mu,G)$-projection
$$
[\phi]:M_{\RR^+\mu}\rightarrow
\left(M\circleddash\operatorname{Diag}\left(\mathcal{O}_
{\RR^+\mu}\times\frac{G}{K_\mu}\right)\right)_0\,\text{,}\,[\phi](\hat{x}):=[x,(-J(x),\hat{e})],
$$
where $[,]$ and $\hat{ }$ denote the $G$ and $K_\mu$-classes,
respectively. This map is well defined. Indeed, let $k$ be an
element of the kernel group of $\mu$. Then,
$[\phi](\widehat{kx})=[kx,(-J(kx), \hat{e})]=[kx,(-k\cdot
J(x),\widehat{k^{-1}})]=[k\cdot(x,(-J(x),\hat{e}))]=[\phi](\hat{x})$,
for any $\hat{x}\in M_{\RR^+\mu}$. To see that $[\phi]$ is
injective, let $\hat{x}_1,\hat{x}_2$ be elements of $M_{\RR^+\mu}$
such that $[x_1,(-J(x_1),\hat{e})]=[x_2,(-J(x_2),\hat{e})]$. Then,
there is $g$ an element of $G$ such that
$(gx_1,(-gJ(x),\hat{g}^{-1})=(x_2,(-J(x_2),\hat{e}))$. It follows
that $g\in K_\mu$ and $gx_1=x_2$. Hence $\hat{x}_1=\hat{x}_2$ and
$[\phi]$ is one-to-one. If $[x,(Ad^*_gr\mu,\hat{g})]$ is an element
of $\left(M\circleddash\operatorname{Diag}\left(\mathcal{O}_
{\RR^+\mu}\times\frac{G}{K_\mu}\right)\right)_0$, then
$J_d(x,(Ad^*_gr\mu,\hat{g}))=J(x)+Ad^*_gr\mu=0$. Therefore, $gx\in
J^{-1}(\RR^+\mu)$,
$$
[\phi](\hat{gx})=[gx,(-J(gx),\hat{e})]=[gx,(-gAd^*_gr\mu,\widehat{gg^{-1}})]=[x,(-J(x),\hat{e})],
$$
and $[\phi]$ is onto. As it is obviously a smooth map, we obtain
that $[\phi]$ is in fact a diffeomorphism with inverse given by
$[x,(Ad^*_gr\mu,\hat{g})]\mapsto [gx].$

To show that $[\phi]$ is also a symplectic map, fix $\hat{x}$ in
$M_{\RR^+\mu}$ and $(v_i)_{i=1,2}$ in $T_x J^{-1}(\RR^+\mu)$. Note
that $T_x (\pi_2\circ\phi)(v_i)$ belongs to $\RR\mu\simeq
T_{J(x)}(\RR^+\mu)$ for each $i=1,2$. Suppose $J(x)=r\mu$ and $(T_x
(\pi_2\circ\phi)(v_i)=r_i\mu)_{i=1,2}$ with $(r_i)_{i=1,2}$ reals.
Then, using the function equalities
$[\phi]\circ\pi_{K_\mu}=\pi_G\circ\phi$ and
$\pi_1\circ\phi=Id_{J^{-1}(\RR^+\mu)}$, we obtain
\begin{align*}
[\phi]^*(\pi^*_1\omega-\pi^*_2\omega^-_{\mathcal{O}_{\RR^+\mu}})_0(\hat{x})(T_x\pi_{K_\mu}
v_1,T_x\pi_{K_\mu}
v_2)=\\
(\pi^*_1\omega-\pi^*_2\omega^-_{\mathcal{O}_{\RR^+\mu}})_0([x,(-J(x),\hat{e})])(T_x
([\phi]\circ\pi_{K_\mu})v_1, T_x ([\phi]\circ\pi_{K_\mu})v_2)=\\
(\pi^*_1\omega-\pi^*_2\omega^-_{\mathcal{O}_{\RR^+\mu}})(\phi(x))(T_x\phi
v_1,T_x\phi v_2)=\\
\omega(x)(T_x (\pi_1\circ\phi)v_1, T_x
(\pi_1\circ\phi)v_2)-\omega^-_{\mathcal{O}_{\RR^+\mu}}(J(x),\hat{e})(T_x(\pi_2\circ\phi)
v_1,T_x(\pi_2\circ\phi) v_2)=\\
i^*_\mu\omega(x)(T_x (i_\mu\circ\pi_1\circ\phi)v_1,T_x
(i_\mu\circ\pi_1\circ\phi)v_2)-\omega^-_{\mathcal{O}_{\RR^+\mu}}(r\mu,\hat{e})(r_1\mu,r_2\mu)=\\
i^*_\mu\omega(x)(T_x (i_\mu\circ\pi_1\circ\phi)v_1,T_x
(i_\mu\circ\pi_1\circ\phi)v_2)=\omega_{\RR^+\mu}(\hat{x})(T_x\pi_{K_\mu}
v_1,T_x\pi_{K_\mu} v_2),
\end{align*}
completing thus the proof of this theorem.
\end{proof}

In the remaining of this section we will study the ray reduced
spaces of the cosphere bundle of the Lie group $G$. Consider the
action of the multiplicative group $\RR^+$ by dilatations on the
fibers of $T^* G\setminus\{0_{T^* G}\}$. The cosphere bundle of $G$,
$S^*G$ is the quotient manifold $(T^* G\setminus\{0_{T^*
G}\})/\RR^+$. Denote by $\pi:T^* G\setminus\{0_{T^*
G}\}\rightarrow S^* (G)$ the canonical projection. Then,
$(\pi, \RR^+, T^* G\setminus\{0_{T^* G}\}, S^*G)$ is a
$\RR^+$-principal bundle. $S^*G$ admits a canonical contact
structure given by the kernel of any one form constructed as the
pull-back of the Liouville form on $T^*G$ through a global section
of the $\RR^+$-principal bundle $(\pi, \RR^+, T^*
G\setminus\{0_{T^* G}\}, S^*G)$. Namely, for every global section
$\sigma:S^*G\rightarrow T^* G\setminus\{0_{T^* G}\}$ the one-form
$\theta_\sigma=\sigma^*\theta$ determines the same contact
structure. Note that $\pi_{\RR^+}^*\theta_\sigma=f_\sigma\theta$,
where $f_\sigma:T^* G\setminus\{0_{T^* G}\}\rightarrow\RR^+$ is a
smooth function with the property that
$f_\sigma(r\alpha_g)=\frac{1}{r}f\sigma(\alpha_g)$ for any
$r\in\RR^+$ and $\alpha_g\in T^*G$. The action by left translations
of $G$ on its cotangent bundle induces a free and proper action on
the copshere bundle given by
$$
g'\cdot\{\alpha_g\}:=\{T^*_{g'g}L_{g^{-1}}\alpha_g\},
$$
for all $\{\alpha_g\}\in S^*G$ and $g'\in G$. Since it is a proper
action which preserves the contact structure, there is always a
global section $\sigma$ such that the action will preserve the
associated contact form $\theta_\sigma$. Then, this action admits an
equivariant momentum map defined by
$$
\langle J_{sL}(\{\alpha_g\}),\xi\rangle
:=\theta_\sigma(\{\alpha_g\}(\xi_{S^*G})(\{\alpha_g\})=f_\sigma(\alpha_g)\alpha_g(\xi_G
(g)),
$$
where $\{\alpha_g\}\in S^*G$ and $\xi\in\mathfrak{g}$. That is,
$J_{sL}(\alpha_g)=f_\sigma(\alpha_g)\alpha_g$, for any
$\{\alpha_g\}\in S^*G$. Here we have briefly recalled the
construction and some of the properties of the cosphere bundle of a
Lie group. For more details the interesting reader is referred to
\cite{dragulete--ornea--ratiu}, \cite{ekholm--etnyre}, and
\cite{ratiu--schmidt}.

Denote by
$\operatorname{Diag}\left(S^*(\mathcal{O}_{\RR^+\mu})\times\frac{G}{K_\mu}\right)$
the diagonal product of the $\pi$-quotient of the ray orbit
of $\mu$ and $\frac{G}{K_\mu}$. The quoteint space
$S^*(\mathcal{O}_{\RR^+\mu})$ is a smooth manifold since the
$\RR^+$-action on $\mathcal{O}_{\RR^+\mu}$ is free and proper. The
map
$$
[g]\in\frac{G}{G_{\RR^+\mu}}\longrightarrow Ad^*_g r\mu
$$
is a diffeomorphism. Define the following one form on
$\operatorname{Diag}\left(S^*(\mathcal{O}_{\RR^+\mu})\times\frac{G}{K_\mu}\right)$
\begin{multline}
\label{forma contact}
 \eta_{\mathcal{O}_{\RR^+\mu}}(\{Ad^*_g
r\mu\},\hat{g})(T_{Ad^*_g
r\mu}\pi_{\RR^+\mu}(ad^*_{\xi}Ad^*_gr\mu+r'Ad^*_g r\mu,
\hat{\xi_G}(\hat{g})):=\\f_\sigma(T^*_gR_{g^{-1}}r\mu)\langle
Ad^*_gr\mu,\xi\rangle,
\end{multline}
for any $(\{Ad^*_g
r\mu\},\hat{g})\in\operatorname{Diag}\left(S^*(\mathcal{O}_{\RR^+\mu})\times\frac{G}{K_\mu}\right)$
and any tangent vector
$$
T_{Ad^*_g r\mu}\pi_{\RR^+\mu}(ad^*_{\xi}Ad^*_gr\mu+r'Ad^*_g r\mu,
\hat{\xi_G}(\hat{g})\in T_{(\{Ad^*_g
r\mu\},\hat{g})}\operatorname{Diag}\left(S^*(\mathcal{O}_{\RR^+\mu})\times\frac{G}{K_\mu}\right).
$$
As we will see in the proof of the following Theorem, the diagonal
manifold \newline
$\left(\operatorname{Diag}\left(S^*(\mathcal{O}_{\RR^+\mu})\times\frac{G}{K_\mu}\right),
\eta_{\mathcal{O}_{\RR^+\mu}}\right)$ is a well defined exact
contact manifold.

\begin{te}
\label{teorema--sase--doi} Let the Lie group $G$ act on its
cosphere bundle $S^*G$ by the lift of left translations on itself.
Suppose $\mu$ is an element of the dual of its Lie algebra with
kernel group $K_\mu$ and the property that
$\ker\mu+\mathfrak{g}_\mu=\mathfrak{g}$, where $\mathfrak{g}_\mu$
is the isotropy algebra of $\mu$ for the coadjoint action. Then
the ray reduced space at $\mu$, $(S^*G)_{\RR^+\mu}$ is well
defined and contactomorphic to
$\left(\operatorname{Diag}\left(S^*(\mathcal{O}_{\RR^+\mu})\times\frac{G}{K_\mu}\right),\eta_{\mathcal{O}_{\RR^+\mu}}\right)$,
where $\eta_{\mathcal{O}_{\RR^+\mu}}$ is the one form define by
\eqref{forma contact}.
\end{te}
\begin{proof}
First note that since the $K_\mu$ and $\RR^+$-actions commute we
have the equality
$J^{-1}_{sL}(\RR^+\mu)=\pi(J^{-1}_L(\RR^+\mu))$. Even more
the maps,
$$
\phi:(S^*G)_{\RR^+\mu}\rightarrow\frac{(T^*G)_{\RR^+\mu}}{\RR^+}\quad\text{and}\quad\bar{J}_{R^+}:\frac{(T^*G)_{\RR^+\mu}}{\RR^+}
\rightarrow\operatorname{Diag}\left(S^*(\mathcal{O}_{\RR^+\mu})\times\frac{G}{K_\mu}\right)
$$
defined by $\phi([\{\alpha_g\}]):=\{[\alpha_g]\}$ and
$\bar{J}_{R^+}(\{[T^*_gR_{g^{-1}}r\mu:=(\{Ad^*_gr\mu\},\hat{g})$,
for any $\alpha_g$ in $J^{-1}_L(\RR^+\mu)$ are diffeomorphisms. Let
$\Psi:(S^*G)_{\RR^+\mu}\rightarrow\operatorname{Diag}\left(S^*(\mathcal{O}_{\RR^+\mu})\times\frac{G}{K_\mu}\right)$
be the map $\Psi:=\bar{J}_{R^+}\circ\phi$. It is obviously a
diffeomorphism with inverse given by
$$
\Psi^{-1}:\operatorname{Diag}\left(S^*(\mathcal{O}_{\RR^+\mu})\times\frac{G}{K_\mu}\right)\rightarrow(S^*G)_{\RR^+\mu}\,\text{,}\,
\Psi^{-1}(\{Ad^*_gr\mu\},\hat{g})=[\{T^*_gR_{g^{-1}}r\mu\}],
$$
for any $g\in G$ and $r\in\RR^+$. Denote by $\eta_{\RR^+\mu}$ the
reduced contact form of $(S^*G)_{\RR^+\mu}$. Then,
\begin{multline}
(\Psi^{-1})^* (\eta_{\RR^+\mu})(\{Ad^*_g r\mu\},\hat{g})(T_{Ad^*_g
r\mu}\pi_{\RR^+\mu}(ad^*_{\xi} Ad^*_g r\mu+r'Ad^*_g
r\mu,\hat{\xi}_G (\hat{g}))=\\
(\pi^s_{\RR^+\mu})^*(\{T^*_gR_{g^{-1}}r\mu\})\left(\left
.\frac{d}{dt}\right |_{t=0} \pi_{\RR^+}(Ad^*_{g\exp t\xi}
e^{tr'}Ad^*_g r\mu, \widehat{g\cdot\exp
t\xi})\right)=\\(\pi^s_{\RR^+\mu})^*(\{T^*_gR_{g^{-1}}r\mu\})\left(T_{T^*_gR_{g^{-1}}r\mu}\pi_{\RR^+}
(X^{\xi}(T^*_gR_{g^{-1}}r\mu)) \right)=\\\theta_\sigma
(\{T^*_gR_{g^{-1}}r\mu\})\left(T_{T^*_gR_{g^{-1}}r\mu}\pi_{\RR^+}
(X^{\xi}(T^*_gR_{g^{-1}}r\mu)) \right)
=\\(\pi^*_{\RR^+}\theta)(T_{T^*_g
R_{g^{-1}}r\mu})(X^{\xi}(T^*_gR_{g^{-1}}r\mu))=f_\sigma(T^*_gR_{g^{-1}}r\mu)\theta
(T^*_gR_{g^{-1}}r\mu)(X^{\xi}(T^*_gR_{g^{-1}}r\mu))\\=f_\sigma(T^*_gR_{g^{-1}}r\mu)J^{\xi}_R
(T^*_gR_{g^{-1}}r\mu)=f_\sigma(T^*_gR_{g^{-1}}r\mu)\langle Ad^*_g
r\mu, \xi\rangle=\\\eta_{\mathcal{O}_{\RR^+\mu}}(\{Ad^*_g
r\mu\},\hat{g})(T_{Ad^*_g
r\mu}\pi_{\RR^+\mu}(ad^*_{\xi}Ad^*_gr\mu+r'Ad^*_g r\mu,
\hat{\xi_G}(\hat{g})),
\end{multline}
for all $\xi\in\mathfrak{g}$ and $g\in G$. Hence,
$\eta_{\mathcal{O}_{\RR^+\mu}}$ is a contact form and $\Psi$ the
required contactomorphism.
\end{proof}

\begin{co}
In the hypothesis of Theorem \ref{teorema--sase--doi}, the contact
form $\eta_{\mathcal{O}_ {\RR^+\mu}}$ defined by \ref{forma contact}
is $G$-invariant with respect to the following action
$$
g_1\cdot (\{Ad^*_g r\mu\},\hat{g}):=\left(\{Ad^*_{g_1^{-1}}Ad^*_g
r\mu\},\widehat{gg_1^{-1}}\right),
$$
for each $g_1$ in $G$ and $({Ad^*_g
r\mu},\hat{g})\in\operatorname{Diag}\left(S^*(\mathcal{O}_
{\RR^+\mu})\times\frac{G}{K_\mu}\right)$.
\end{co}

\section{Ray quotients of K\"ahler and Sasakian-Einstein Manifolds}
\label{sectiunea--sapte}

In this section we will study the behavior of K\"ahler-Einstein
metrics of positive Ricci curvature with respect to symmetries.
Namely, in the hypothesis of Theorem \ref{reducere--doi} and using
techniques  developed in \cite{futaki--unu} and \cite{futaki--doi}
by A. Futaki, if $M$ is a Fano manifold and $\omega$ represents
its first Chern class we will show how to compute the Ricci form
of the reduced space in terms of the reduced K\"ahler form
$\omega_{\RR^+\mu}$ and data on $J^{-1}(\RR^+\mu)$ and the kernel
group $K_\mu$. As a corollary we will obtain that if $M$ is a Fano
manifold and the symplectic ray reduction of Theorem
\ref{reducere--unu} can be performed, then the ray reduced
symplectic manifold $M_{\RR^+\mu}$ will also be Fano. Even more,
if $M$ is a compact K\"ahler-Einstein manifold of positive Ricci
curvature, then $M_{\RR^+\mu}$ is Einstein if and only if the norm
of a certain multi vector field defined using the kernel algebra
$\mathfrak{k}_\mu$ and the algebra $\mathfrak{m}$ defined in
\eqref{descompunere--m} is constant on $J^{-1}(\RR^+\mu)$.

Recall that the Ricci form $\rho$ of a compact K\"ahler manifold
$(M,\operatorname{g},\omega)$ is a real closed $(1,1)$-form whose
class in the de Rham cohomology group $H^2_{DR}(M)$ defines the
first Chern class of the manifold. Suppose that the K\"ahler form
$\omega$ represents the first Chern class of
$M$. Then, applying the local $i\partial\bar{\partial}$-Lemma
(see, for instance \cite{moroianu}), we obtain that there is a
smooth real function $f$ such that
$\rho-\omega=\frac{\sqrt{-1}}{2\pi}
\partial\bar{\partial}f$. If the compact Lie group $G$ acts on $M$ by
holomorphic isometries, then there is always an associated
equivariant momentum map $J:M\rightarrow\mathfrak{g}^*$. We will
now recall its construction.

By Theorem $2.4.3$ in \cite{futaki--unu} there is an isomorphism
between the complex Lie algebra of holomorphic vector fields on
$M$ and the set of all complex-valued functions $u$ satisfying
$\Delta_f u-u=0$. This isomorphism is given by $u\mapsto
\operatorname{grad}u$. Here, $\Delta_f$ is the differential
operator given by

$$u\mapsto \Delta u-\nabla^i u\nabla_i f=\Delta
u-\operatorname{g}^{i\bar{j}}\frac{\partial u}{\partial
\bar{z}_j}\partial_{z_i}f=\Delta u-\operatorname{grad}u(f),$$

with $\Delta$ the complex Laplacian, $\nabla$ the covariant
derivative associated to $\operatorname{g}$ and $(z_i)_i$ local
holomorphic coordinates. Then, the infinitesimal isometries
associated to the elements of the Lie algebra $\mathfrak{g}$ embed
in the space of holomorphic vector fields on $M$ as follows:
assign to each $\xi\in\mathfrak{g}$ the holomorphic vector field
$\xi'_M:=\frac{1}{2}(\xi_M-\sqrt{-1}\mathcal{C}_{\operatorname{g}}\xi_M)$.
$\mathcal{C}$ denotes the complex structure of
$(M,\operatorname{g})$. In other words, all the infinitesimal
isometries are real holomorphic vector fields. Therefore, there is
a smooth complex function $u_{\xi'_M}$ with
$\operatorname{grad}u_{\xi'_M}=\xi'_M$.

\begin{lm}
For every $\xi$ element of the Lie algebra $\mathfrak{g}$, the
above defined function $u_{\xi'_M}$ is purely imaginary.
\end{lm}
\begin{proof}
Since $G$ acts by holomorphic isometries, $\xi_M(f)=0$ and we have
$$
\Delta_f u_{\xi'_M}=\Delta u_{\xi'_M}-\xi'_M (f)=\Delta
u_{\xi'_M}+\frac{\sqrt{-1}}{2}\mathcal{C}(\xi_M)f\,\text{and}\,\bar{\Delta_f
u_{\xi'_M}}=\bar{\Delta
u_{\xi'_M}}-\frac{\sqrt{-1}}{2}\mathcal{C}(\xi_M)f.
$$
Using the fact that $\Delta_f u_{\xi'_M}-u_{\xi'_M}=0$ it follows
that
\begin{equation}
\label{ecuatie--laplasian} \Delta
(u_{\xi'_M}+\bar{u_{\xi'_M}})=u_{\xi'_M}+\bar{u}_{\xi'_M}.
\end{equation}
On the other hand, it is well known that on a complex connected
Riemannian manifold, if $X$, the gradient of a function $u$ is a
holomorphic vector field, then it is a Killing vector field if and
only if $u+\bar{u}$ is constant. In particular, if $u$ is purely
imaginary, then the real part of $X$ is a Killing vector field. For
a proof of this, see for instance \cite{calabi}. Applying this to
$X=\operatorname{grad}(u_{\xi'_M}+\bar{u}_{\xi'_M})$ we obtain
that $u_{\xi'_M}+\bar{u}_{\xi'_M}$ is a constant function. Hence,
\eqref{ecuatie--laplasian} implies that
$u_{\xi'_M}+\bar{u}_{\xi'_M}=0$.
\end{proof}

\begin{pr}
Let $M$ be a compact complex manifold of positive first Chern
class and dimension $n$. Choose any K\"ahler metric
$\operatorname{g}$ which represents the first Chern class and
suppose the Lie group $G$ acts on $(M,\operatorname{g})$ by
holomorphic isometries. Then the map
$J:M\rightarrow\mathfrak{g}^*$, $\langle
J(x),\xi\rangle:=\frac{\sqrt{-1}}{2\pi}u_{\xi'_M}$ defines an
equivariant momentum map for the action of $G$ on $M$.
\end{pr}

\begin{proof}
In a local holomorphic coordinate system $(z_1,...,z_n)$, the
K\"ahler form associated to $\operatorname{g}$ is given by
$\omega=\frac{\sqrt{-1}}{2\pi}\operatorname{g}_{\alpha\bar{\beta}}dz^\alpha\wedge
dz^{\bar{\beta}}$. Then, for any $\xi$ in $\mathfrak{g}$,
$$
i_{\xi'_M}\omega=i_{\operatorname{grad}
u_{\xi'_M}}\omega
=\frac{i}{2\pi}\operatorname{g}^{\alpha\bar{\beta}}\nabla_{\bar{\beta}}u_{\xi'_M}\operatorname{g}_{\alpha\bar{\gamma}}d\bar{z}^{\gamma}=
\bar{\partial}J^\xi.
$$
and
$$
i_{\xi_M}\omega=i_{(\xi'_M+\bar{\xi}'_M)}\omega_{\operatorname{g}}=
i_{\xi'_M}\omega+\overline{i_{\xi_M}\omega}=d
J^\xi,
$$
proving thus that $J$ is a momentum map. To show the equivariance
of $J$, fix $g\in G$ and $\xi\in\mathfrak{g}$. Observe that the
$G$-action commutes with the operator $\Delta_f$ and that for any
vector field $Y$ of type $(0,1)$ we have
\begin{align*}
\omega(\operatorname{grad}(g^*u_{\xi'_M}),Y)=Y(g^*J^\xi)= (g_*
Y)J^\xi=\omega(\operatorname{grad}u_{\xi'_M},
g_*Y)=\\
\omega(g^{-1}_*\xi'_M,Y)=\omega((ad_{g^{-1}}\xi)'_M,
Y)=\omega(\operatorname{grad} u_{(ad_{g^{-1}}\xi)'_M}, Y).
\end{align*}
Hence, $J$ is also $G$-equivariant.
\end{proof}

Assume that the hypothesis of Theorem \ref{reducere--doi} are
verified for a momentum value $\mu$. Choose $\{\xi_i\}_{i=1,k}$
and $\{\eta_i\}_{i=1,m}$ basis of $\mathfrak{k}_\mu$ and
$\mathfrak{m}$ such that the associated infinitesimal isometries
form an orthogonal frame of the vertical distribution of
$\pi_{\RR^+\mu}:J^{-1}(\RR^+\mu)\rightarrow M_{\RR^+\mu}$ and of
$\mathfrak{m}_M$ respectively. Recall that $\mathfrak{m}_M$ is the
space defined in the decomposition \eqref{eq4}. Denote by
$\xi'\wedge\eta'$ the multi vector
$\xi'_{1M}\wedge...\wedge\xi'_{kM}\wedge\eta'_{1M}\wedge...\wedge\eta'_{mM}$.
We are now ready to state the main theorem of this section.

\begin{te}
\label{curbura--Ricci} Let $(M, \operatorname{g}, \omega)$ be a
Fano K\"ahler manifold with $\omega$ representing its first Chern
class. Let $G$ be a Lie group acting on $M$ by holomorphic
isometries. Suppose that $\mu$ is an element of the dual of the
Lie algebra of $G$ such that the ray reduced space is a well
defined K\"ahler orbifold $(M_{\RR^+\mu},\omega_{\RR^+\mu})$
Assume that the kernel group $K_\mu$ is compact. Then, the Ricci
form of the ray reduced space is given by
\begin{equation}
\label{ricci}
\rho_{\RR^+\mu}=\omega_{\RR^+\mu}+\frac{\sqrt{-1}}{2\pi}\partial\bar{\partial}(f_{\RR^+\mu}+\operatorname{log}\Vert
\xi'\wedge\eta'\Vert^2_{\RR^+\mu}),
\end{equation}
where $f_{\RR^+\mu}$ and $\Vert \xi'\wedge\eta'\Vert_{\RR^+\mu}$
are the $K_\mu$-projections of $f$ and the point-wise norm of the
multi vector $\xi'\wedge\eta'$. Consequently, $M_{\RR^+\mu}$ is
also Fano.
\end{te}

\begin{proof}
First note that $\xi'\wedge\eta'$ is $K_\mu$-invariant. For any
$g\in G$ and $x\in M$, $ \xi_M(gx)=g_*(ad_{g^{-1}}\xi)_M(x)$.
Since the kernel group $K_\mu$ is compact,
$\det(ad_{g^{-1}}|_{\mathfrak{k}_\mu})=1$ and
$$
(\xi'_{1M}\wedge...\wedge\xi'_{kM})(gx)=(\det(ad_{g^{-1}}|_{\mathfrak{k}_\mu})g_*(\xi'_{1M}\wedge...\wedge\xi'_{kM})(x)=
g_*(\xi'_{1M}\wedge...\wedge\xi'_{kM})(x).
$$
Even more
$\det(ad_{g^{-1}}|_{\mathfrak{g}})=\det(ad_{g^{-1}}|_{\mathfrak{g}_\mu})\det(ad_{g^{-1}}|_{\mathfrak{m}})$
and since $G_\mu$ and $G$ are compact, it follows that
$\det(ad_{g^{-1}}|_{\mathfrak{m}})=1$. By an argument similar to
the one above we obtain that the multi vector $\xi'\wedge\eta'$ is
$K_\mu$-invariant.

 The action of $K_\mu$ being
by isometries it is clear that the point wise norm of
$\xi'\wedge\eta'$ is also $K_\mu$-invariant. Recall from Theorem
\ref{reducere--doi} that we have the following orthogonal
decomposition:
\begin{equation}
T_x
M=\mathcal{V}_x\oplus\mathcal{H}_x\oplus\mathfrak{m}_{M}(x)\oplus\mathcal{C}(\mathcal{V}_x).
\end{equation}
$\mathcal{V}_x$ is the vertical space at $x$ of the Riemannian
submersion $\pi_{\RR^+\mu}:J^{-1}(\RR^+\mu)\rightarrow
M_{\RR^+\mu}$ and it is generated by $\{\xi_{iM}(x)\}_{i=1,k}$.
The horizontal space at $x$ is $\mathcal{H}_x$ and
$\mathfrak{m}_{M}(x)$ is invariant with respect to the complex
structure $\mathcal{C}$. Let $\mathcal{V}$ and $\mathcal{M}$ be
the distributions defined by
$\{\mathcal{V}_x\oplus\mathcal{C}(\mathcal{V}_x)\}_{x\in
J^{-1}(\RR^+\mu)}$ and $\{\mathfrak{m}_{M}(x)\}_{x\in
J^{-1}(\RR^+\mu)}$. Consider the following decompositions

\begin{equation*}
\begin{array}{ll}
\mathcal{V}\otimes\CC & = \mathcal{V}^{1,0}\oplus\mathcal{V}^{0,1}  \\
\mathcal{H}\otimes\CC & = \mathcal{H}^{1,0}\oplus\mathcal{H}^{0,1}  \\
\mathcal{M}\otimes\CC& = \mathcal{M}^{1,0}\oplus\mathcal{M}^{0,1}.
\end{array}
\end{equation*}
Then we have that $i^*_{\RR^+\mu}T^{1,0}M =
\mathcal{H}^{1,0}\oplus\mathcal{V}^{1,0}\oplus\mathcal{M}^{1,0}$.
Denote by $\nabla^h$, $\nabla^v$, $\nabla^m$, and
$\nabla^{\RR^+\mu}$ the connections induced by the Levi-Civita
connection of $M$ on $\mathcal{H}^{1,0}$, $\mathcal{V}^{1,0}$,
$\mathcal{M}^{1,0}$, and $i^*_{\RR^+\mu}T^{1,0}M$ (or their
determinant bundles). Let $\theta^h$, $\theta^v$, $\theta^m$, and
$\theta^{\RR^+\mu}$ be the connection forms of the above defined
connections with respect to the local, orthogonal and $K_\mu$-
invariant frames $Y_1\wedge...\wedge Y_s$,
$\xi'_{1M}\wedge...\wedge\xi'_{kM}$,
$\eta'_{1M}\wedge...\wedge\eta'_{mM}$, $Y_1\wedge...\wedge
Y_s\wedge\xi'_{1M}\wedge...\wedge\xi'_{kM}\wedge\eta'_{1M}\wedge...\wedge\eta'_{mM}$,
respectively. Then,
$\theta^{\RR^+\mu}=\theta^h+\theta^v+\theta^m$. Extend the
connection forms by
\begin{align*}
\theta^h_h(Y)&=\theta^h(Y)   & \theta^h_h(\xi_M)&=0&   &
\theta^h_h(\eta_M)&=0& &\theta^h_v(Y)&=0 &
\theta^h_v(\xi_M)&=\theta^h(\xi_M)\\
\theta^v_h(Y)&=\theta^v(Y)             & \theta^v_h(\xi_M)&=0&   &
\theta^v_h(\eta_M)&=0& &\theta^v_v(Y)&=0 &
\theta^v_v(\xi_M)&=\theta^v(\xi_M)\\
\theta^m_h(Y)&=\theta^m(Y)             & \theta^m_h(\xi_M)&=0&
&\theta^m_h(\eta_M)&=0& &\theta^m_v(Y)&=0 &
\theta^m_v(\xi_M)&=\theta^m(\xi_M)
\end{align*}

\begin{align*}
\theta^h_v(\eta_M)&=0   & \theta^h_m(Y)&=0&   &
\theta^h_m(\xi_M)&=0&  &\theta^h_m(\eta_M)&=\theta^h(\eta_M) & & &
&
 & \\
\theta^v_v(\eta_M)&=0             & \theta^v_m(Y)&=0&   &
\theta^v_m(\xi_M)&=0&   &\theta^v_m(\eta_M)&=\theta^v(\eta_M)& & &
&
 & \\
\theta^m_v(\eta_M)&=0             & \theta^m_m(Y)&=0&  &
\theta^m_m(\xi_M)&=0&  &\theta^m_m(\eta_M)&=\theta^m(\eta_M)& & &
& &,
\end{align*}
for any $Y\in\mathcal{H}$, $\xi_M\in\mathcal{V}$, and
$\eta_M\in\mathcal{M}$. Then $\theta=\theta^h_h+B$, where
$B=\theta^h_v+\theta^h_m+\theta^v_h+\theta^v_v+\theta^v_m+
\theta^m_h+\theta^m_v+\theta^m_m$. Finally, let
$\theta_{\RR^+\mu}$ be the connection form of the fiber bundle
$\det T^{1,0}M_{\RR^+\mu}$ with respect to the local orthogonal
frame $\pi_{\RR^+\mu*}Y_1\wedge...\wedge\pi_{\RR^+\mu*}Y_s$.

We want to prove that
$(\pi_{\RR^+\mu})^*\theta_{\RR^+\mu}=\theta^h_h$. First,note that
the Levi-Civita connection of $M_{\RR^+\mu}$ is given by
$$
\nabla^{\RR^+\mu}_{\hat{X}_1}
\hat{^X}_2=\pi_{\RR^+\mu*}(hor(\nabla_{X_{1h}} X_{2h})),
$$
for any $\hat{X}_1$, $\hat{X}_2$ vector fields on the quotient.
Here $hor$ denotes the horizontal projection and $X_{1h}$,
$X_{2h}$ are the unique sections of the horizontal distribution
which project onto $\hat{X}_1$ and $\hat{X}_2$. Then, we obtain
\begin{equation*}
\begin{split}
(\pi^*_{\RR^+\mu}\theta_{\RR^+\mu})(X_h)=\nabla^{\RR^+\mu}_{\pi_{\RR^+\mu_*}
X_h}(\pi_{\RR^+\mu*}(Y_1)\wedge...\wedge\pi_{\RR^+\mu*}(Y_k))=\\
{\sum}_{i=1}^{k}
\pi_{\RR^+\mu*}(Y_1)\wedge...\wedge\pi_{\RR^+\mu*}(hor\nabla_{X_h}
Y_i)\wedge...\wedge\pi_{\RR^+\mu*}(Y_k)\\=\pi_{\RR^+\mu}*(\nabla_{X_h}
(Y_1\wedge...\wedge
Y_k))=\pi_{\RR^+\mu}*(\theta^h_h(X_h)Y_1\wedge...\wedge Y_k).
\end{split}
\end{equation*}
Since the frame is $K_\mu$-invariant it follows that
$\theta^h_h(X_h)$ is a $K_\mu$-invariant function on
$J^{-1}(\RR^+\mu)$, for any horizontal vector field $X_h$.
Therefore, $(\pi_{\RR^+\mu})^*\theta_{\RR^+\mu}=\theta^h_h$ and
\begin{equation}
\label{aoleu--doamne--ajutor}
\pi_{\RR^+\mu}^*\rho_{\omega_{\RR^+\mu}}=\frac{\sqrt{-1}}{2\pi}d\pi_{\RR^+\mu}^*\theta_{\RR^+\mu}=\frac{\sqrt{-1}}{2\pi}d\theta^h_h
=\frac{\sqrt{-1}}{2\pi}(d\theta-B)
\\
=i^*_{\RR^+\mu}\rho_{\omega}-\frac{\sqrt{-1}}{2\pi}B,
\end{equation}
where
$B:=d\theta^h_v+d\theta^h_m+d\theta^v_h+d\theta^v_v+d\theta^v_m+d\theta^m_h+d\theta^m_v+d\theta^m_m$.

Observe that
\begin{equation}
\label{puf--puf} d\theta^v_h=d\pi_{\RR^+\mu}^*(\partial
\log\Vert\xi'\Vert^2_{\RR^+\mu})=\pi_{\RR^+\mu}^*(\bar{\partial}\partial\log
\Vert\xi'\Vert^2_{\RR^+\mu}).
\end{equation}
Indeed, fix $Y$ a section of $\mathcal{H}^{1,0}$. Working in
holomorphic coordinates it is very easy to see that
$\nabla^v_{\bar{Y}}\xi'_M=0$. On the other hand

\begin{equation}
\begin{split}
\nabla^v_Y\xi'={\sum}_{i=1}^{k}\xi'_{1M}\wedge...\wedge\frac{\operatorname{g}(\nabla_Y\xi'_{iM},\bar{\xi}'_{iM}
)}{\Vert\xi'_{iM}\Vert^2}\xi'_{iM}\wedge...\wedge\xi'_{kM}=\\
{\sum}_{i=1}^{k}Y(log(\Vert\xi'_{iM}\Vert^2))\xi'=Y(log(\Vert\xi'\Vert^2))\xi'.
\end{split}
\end{equation}
Hence formula \eqref{puf--puf} follows. In a similar way we can
see that $d\theta^m_h=\pi^*_{\RR^+\mu}(\bar{\partial}\partial
 (log\Vert\eta'\Vert^2))$. Therefore,
\begin{equation}
\label{ecuatie--sapte}
d\theta^v_h+d\theta^m_h=\pi^*_{\RR^+\mu}(\bar{\partial}\partial\,
 log\Vert\xi'\wedge\eta'\Vert^2).
\end{equation}

Applying Lemma $7.3.8$ in \cite{futaki--unu}, we know that for any
$\gamma$, section of $\det T^{1,0}M$ and any $\xi$ in
$\mathfrak{k}_\mu$, $L_{\xi_M}\gamma=\nabla_{\xi_M}\gamma-(2\pi
\sqrt{-1}\Delta J^{\xi})\gamma$. In particular, for
$\gamma:=Y_1\wedge...\wedge
Y_s\wedge\xi'_{1M}\wedge...\wedge\xi'_{kM}\wedge\eta'_{1M}\wedge...\wedge\eta'_{mM}$,
along $J^{-1}(\RR^+\mu)$ we get
$\nabla_{\xi_M}\gamma=L_{\xi_M}\gamma+(2\pi \sqrt{-1}\Delta
J^{\xi})\gamma=-(\xi'_M f)\gamma=$ and
$\nabla_{\eta_M}\gamma=-(\eta'_M f)\gamma$. Recall that from the
definition of $J$ we have that $u_{\xi'_M}=u_{\eta'_M}=0$, for all
$\xi\in\mathfrak{k}_\mu$ and $\eta\in\mathfrak{m}$. Let
$\theta_v:=\theta^h_v+\theta^v_v+\theta^m_v$ and
$\theta_m:=\theta^h_m+\theta^v_m+\theta^m_m$. From the above
computations we have that $\theta_v(\xi_M)=-\xi'_M(f)$ and
$\theta_m(\eta_M)=-\eta'_M(f)=0$, for all $\xi\in\mathfrak{k}_\mu$
and all $\eta\in\mathfrak{m}$. Notice that for the last equality
we have used the fact that $\mathfrak{m}_M$ is invariant with
respect to the complex structure $\mathcal{C}$.

The definitions of $\theta_v$ and $\theta_m$ imply that
\begin{equation}
\label{ecuatie--cinci} \theta_v=-i^*_{\RR^+_mu}\partial
f+\pi^*_{\RR^+_mu}\partial f_{\RR^+_mu}\quad
d\theta_v=i^*_{\RR^+_mu}\partial\bar{\partial}f-\pi^*_{\RR^+_mu}\partial
\bar{\partial}f_{\RR^+_mu},
\end{equation}
\begin{equation}
\label{ecuatie--sase}
d\theta_m=0.
\end{equation}
>From (\ref{aoleu--doamne--ajutor}), (\ref{ecuatie--sapte}),
(\ref{ecuatie--sase}), and (\ref{ecuatie--cinci}), the conclusion
of the theorem follows.
\end{proof}

Theorem \ref{reducere--doi}, Proposition \ref{boothbywang}, and
Theorem \ref{curbura--Ricci} entail the following corollary.
\begin{co}
\label{einstein}
In the hypothesis of Theorem \ref{reducere--doi}, suppose $M$ is
also K\"ahler-Einstein of positive Ricci curvature. Then ray
reduced space $M_{\RR^+\mu}$ is K\"ahler-Einstein if and only if
$\Vert \xi'\wedge\eta'\Vert_{\RR^+\mu}$ is constant on
$J^{-1}(\RR^+\mu)$.
\end{co}
\begin{proof}
It is just a matter of definitions.
\end{proof}

Theorems $2.5$ in \cite{boyer-galicki1} and \ref{curbura--Ricci}
imply

\begin{co}
In the hypothesis of Theorem \ref{curbura--Ricci}, if $M$ has Ricci
curvature strictly bigger then $-2$, then so does $M_{\RR^+\mu}$.
\end{co}

\begin{exs}
We will now show that all the reduced spaces obtained in some examples
of \cite{dragulete--ornea} are in fact Sasakian-Einstein manifolds.
In Example $3.2$ of this article we let the torus $T^2$ act on the sphere
$S^7$ as follows:
$$
((e^{\sqrt{-1}t_0}, e^{\sqrt{-1}t_1}), z)\mapsto (e^{-\sqrt{-1}t_0}z_0, e^{\sqrt{-1}t_0}z_1, e^{\sqrt{-1}t_1}z_2, e^{\sqrt{-1}t_1}z_3).
$$
Recall that the infinitesimel generator is given by:
\begin{equation*}
\begin{split}
(r_1,r_2)_{S^7}(z)&=r_{1}(y_{0}\partial_{x_{0}}-x_{0}\partial_{y_{0}}) +
r_{1}(-y_{1}\partial_{x_{1}}+x_{1}\partial_{y_{1}})\\
&+r_{2}(-y_{2}\partial_{x_{2}}+x_{2}\partial_{y_{2}})+
r_{2}(-y_{3}\partial_{x_{3}}+x_{3}\partial_{y_{3}}),
\end{split}
\end{equation*}
for any $(r_1,r_2)$ in the Lie algebra of $T^2$
and the momentum map by
$J(z)=\langle(\vert z_1\vert^2-\vert z_0\vert^2,\vert z_2\vert^2+\vert
 z_3\vert^2),\cdot\rangle,$
for any $z\in S^7$. For $\mu:=\langle v,\cdot\rangle$, $v=(,1,1)$ we have that
$M_{\RR^+\mu}$ can be diffeomorphicaly identified with 
$S^5(\frac{1}{\sqrt 2})\setminus \mathrm{pr}\left\{ z\in  S^7\mid
\vert z_0\vert^2 =\frac 12\right\}\simeq
S^5(\frac1{\sqrt2})\setminus S^1(\frac
 1{\sqrt2})$,
where $\mathrm{pr}:\CC^4\rightarrow\CC^3$,
$\mathrm{pr}(z_0,\ldots, z_3)=(z_0, z_2,z_3)$. Since the group is commutative
the algebra $\mathfrak{m}$ defined in \eqref{descompunere--m} can be identified with $\{0\}$
and the multivector field of Corollary \ref{einstein} turns out to be a simple vector field $\xi'=(-r,r)'_{S^7}=(-r,r)_{S^7}-\sqrt{-1}\bar{(-r,r)_{S^7}}$, with $r$ any non zero real number. A simple calculation shows that $\Vert(-r,r)'_{S^7}\Vert(z)=\vert r \vert \Vert z\Vert=\vert r \vert$ for all $z\in J^{-1}(\RR^+\mu)$. Hence $\Vert(-r,r)'_{S^7}\Vert$ it is constant on $J^{-1}(\RR^+\mu)$ and $M_{\RR^+\mu}$ is a Sasakian-Einstein manifold.

In Example $3.4$ of the same article, a new Sasakian manifold is obtained for $\mu$ defined exactly as above and the initial action on $S^7$ weighted into
$$
((e^{it_0},e^{it_1}),z)\mapsto
(e^{it_{0}\la_0}z_0,e^{it_{1}\la_1}z_1, z_2, z_3),
$$
with $\lambda_0$, $\lambda_1$ positive constants. This time, the norm of $\xi'(z_0,z_1,z_2,z_3)=(-1,1)'_{S^7}(z_0,z_1,z_2,z_3)$ equals 
$\sqrt{2}\vert\lambda_1\vert\Vert z_1\Vert$ which is not constant on $J^{-1}(\RR^+\mu)=S^7\cap (\CC^*\times A)$ with $A$ is the ellipsoid of equation 
$$
\vert z_1\vert^2\left(1+\frac{\lambda_1}{\lambda_0}\right)+\vert z_2\vert^2+\vert z_3\vert^2=1.
$$
Therefore, the reduced Sasakian manifold 
$$M_{\RR^+\mu}=\bigcup_{(z_2,z_3)\in
 \mathrm{pr}(J^{-1}(\RR^+\mu))}S^1(\be^{-\la_0}\al^{\la_1})\times\left\{(z_2,z_3)
\right\}$$
where $\mathrm{pr}:\CC^4\rightarrow\CC^2$,
$\mathrm{pr}(z_0,\ldots, z_3)=(z_2,z_3)$, $\be=\sqrt{\frac{\la_0(1-\vert
z_2\vert^2-\vert z_3\vert^2)}{\la_0+\la_1}}$, and
$\al=\sqrt{\frac{\la_1(1-\vert z_2\vert^2-\vert
 z_3\vert^2)}{\la_0+\la_1}}$ is not Einstein.

On the other hand, the new contact structure obtained in \cite{dragulete--ornea--ratiu}, Example $3.5$ is Sasakian-Einstein since
the infinitesimal isometries generated by the kernel algebra of $\mu$ are independent of the configuration points.
\end{exs}

\noindent {\bf Acknowledgment.} For valuable discussions and suggestions I want to thank Liviu Ornea, Juan-Pablo Ortega, and Tudor Ratiu.


\begin{thebibliography}{99}

\bibitem{arnold} V. I. Arnold, \emph{Sur la g\'eom\'etrie
diff\'erentielle des groupes de Lie de dimension infinie et ses
applications \`a l'hydrodynamique des fluids parfaits}, Ann. Ins.
Fourier, Grenoble, {\bf 16} (1966), 319--361.

\bibitem{besse} A. L. Besse, \emph{Einstein Manifolds},
Springer-Verlag, Berlin, (1971).

\bibitem {blair} D.E.  Blair, \emph{Riemannian geometry of contact and
symplectic manifolds}, Progress in Math. {\bf 203}, Birkh{\"a}user,
Boston, Basel, 2002.

\bibitem{boyer-galicki1} C. P. Boyer, K. Galicki, \emph{On
Sasakian-Einstein Geometry}, Internat. J. Math. {\bf 11} (2000), no.
7, 873--909.

\bibitem{boyer--galicki2} C. P. Boyer, K. Galicki, \emph{$3$-Saskian
Manifolds. Surveys in differential geometry: essays on Einstein
manifolds}, Surv. Diff. Geom., VI, Int. Press, Boston, MA, (1999),
123--184.

\bibitem{bryant} R. L. Bryant, \emph{An Introduction to Lie Groups
and Symplectic Geometry}. Geometry and quantum field theory(Park
City, UT, 1991), 5--181, IAS/Park City Math. Ser., {\bf 1}, Amer.
Math. Soc., Providence, RI, 1995.

\bibitem{calabi} E. Calabi, \emph{Extremal K\"ahler Metrics II,
Differential Geometry and Complex Analysis} (I. Chavel, H. M.
Farkas, eds), Springer-Verlag, Berlin-Heidelberg-New York-Tokyo,
1985, 95--114.

\bibitem{dragulete--ornea}O. M. Dr\u agulete, L. Ornea, \emph{Non-zero contact and Sasakian
reduction}, Diff. Geom. Appl., {\bf 24} (2006), no. 3, 260--270.

\bibitem{dragulete--ornea--ratiu} O. M. Dr\u agulete, L. Ornea, T. S.
Ratiu, \emph{Cosphere Bundle Reduction in Contact Geometry}, J.
Symplectic Geom., {\bf 1} (2003), 695--714.

\bibitem{duistermaat--kolk} J. J. Duistermaat, J. A. Kolk, \emph{Lie
Groups}, (1999), Universitext, Springer-Verlag.

\bibitem{ekholm--etnyre} T. Ekholm, J. B. Etnyre, Invariants of Knots,
Embeddings and Immersions via Contact Geometry, math.GT/0412517.


\bibitem{futaki--unu} A. Futaki, \emph{K\"ahler-Einstein Metrics and
Integral Invariants}, (1988), Lecture Notes in Mathematics, {\bf
1314}, Springer-Verlag.

\bibitem{futaki--doi} A. Futaki, \emph{The Ricci Curvature of
Symplectic Quotients of Fano Manifolds}, Tohoku Math. Journ., {\bf
39} (1987), 329--339.

\bibitem{guillemin--sternberg} V. Guillemin, S. Sternberg, \emph{Homogenous Quantization and Multiplicities
of Group Representations}, J. Funct. Anal., {\bf 47} (1982), 344--380.

\bibitem{hatakeyama} Y. Hatakeyama, \emph{Some notes on
differentiable manifolds with almost contact structures}, T\^ohuku
Math. J. {\bf 15} (1963), 176--181.

\bibitem{iskovskikh--prokhorov} V. A. Iskovskikh, Y. G. Prokhorov,
\emph{Fano varieties}, in Algebraic geometry, V, Encyclop. Math.
Sci. {\bf 47}, Springer, Berlin (1999), 1--247.

\bibitem{kamishima--ornea} Y. Kamishima, L. Ornea, \emph{Geometric
flow on compact locally conformal K\"ahler manifolds}, Tohoku Math.
J.,(2) {\bf 57}, (2005), no. 2, 201--221.

\bibitem{kirillov} A. A. Kirillov, \emph{Elements of the Theory of
Representations}, Grundlehren der mathematischen Wissenschaften,
{\bf 220}, Springer-Verlag, 1976.

\bibitem{kirillov--doi} A. A. Kirillov, \emph{The Orbit Method, I:
Geometric Quantization}, Contemporary Mathematics, vol. {\bf 145},
(1983), 1--63.

\bibitem{kostant--unu} B. Kostant, \emph{Orbits, symplectic structures
and representation theory}, Proc. US-Japan Seminar on Diff. Gem.,
Kyoto, Nippon Hyronsha, Tokyo, {\bf 77} (1965).

\bibitem{kostant--doi} B. Kostant, \emph{On differential geometry
and homogenous spaces II}, Proc. N. A. S. U. S. A., {\bf 42},(1956)
354--357.

\bibitem{mclachlan--perlmutter} R. McLachan, M. Perlmutter,
\emph{Conformal Hamiltonian Systems}, J. Geom. Phys., {\bf 39}
(2001), 276--300.

\bibitem{marsden--ratiu} J. E. Marsden, T. S. Ratiu,
\emph{Introduction to Mechanics and Symmetry}, second edition
(1999), Texts in Applied Mathematics, {\bf 17}, Springer-Verlag.

\bibitem{marsden--weinstein--unu} J. E. Marsden, A. Weinstein,
\emph{Reduction of symplectic manifolds with symmetris}, Rep. Math.
Phys. {\bf 5}(1974), 121--130.

\bibitem{marsden--weinstein--doi} J. E. Marsden, A. Weinstein,
\emph{Comments on the Hystory, Theory, and Applications of the
Symplectic Reduction}, Progr. Math., vol. {\bf 198}, Birkh\"auser
Boston, Boston, MA, 2001.

\bibitem{moroianu} A. Moroianu, \emph{Lectures on K\"ahler
geometry}, London Mathematical Society Student Texts, {\bf
69}(2007), Cambridge University Press.

\bibitem{ornea--verbitsky} L. Ornea, M. Verbistky, \emph{Immersion
theorem for compact Vaisman manifolds}, Math. Ann., {\bf 332},
(2005), no. 1, 121--143.

\bibitem{ortega--ratiu} J-P. Ortega and T.R. Ratiu, \emph{Momentum Maps and
Hamiltonian Reduction}, Progress in Mathematics, Volume {\bf222},
Birkh\"auser, Boston, 2004.

\bibitem{palais} R. S. Palais, \emph{On the Existence of Slices for
actions of Non-Compact Lie Groups}, Ann. Math., {\bf 73} (1961),
295--323.

\bibitem{ratiu--schmidt} T. Ratiu, R. Schmid,
The differentiable structure of three remarkable diffeomorphism
groups, \emph{Math. Z.}, {\bf 177} (1981), 81-100.


\bibitem{sasaki--unu} S. Sasaki, \emph{On differentiable manifolds with
certain structures which are closely related to almost contact
structure}, T\^ohuku Math. J. {\bf 2} (1960), 459--476.

\bibitem{sasaki--doi} S. Sasaki, \emph{Contact structures on
Brieskorn manifolds} (lecture Japan Mathematical Society 1975),
Shiego Sasaki Selected Papers, Kinokuniya, Tokyo, 349--363.

\bibitem{souriau} J-M. Souriau, \emph{Structure des Syst\`emes
Dynamiques}, Dunod. Paris. English translation by R. H. Cushman and
G. M. Tuynman as \emph{Structure of Dynamical Systems. A Symplectic
View of Physics}, {\bf 149}, Prog. Math. Birkh\"auser, 1997.

\bibitem{stefan} P. Stefan, \emph{Accessible sets, orbits, and foliations with
singularities}, Proc. Lond. Math. Soc., {\bf 29}(1974), 699--713.

\bibitem{sussmann} H. Sussmann, \emph{Orbits of families of vector
fields and integrability of distributions}, Trans. Amer. Math. Soc.,
{\bf 180 }(1973), 171--188.

\bibitem{willett} C. Willett, \emph{Contact reduction},  Trans. Amer. Math.
 Soc., {\bf 354}  (2002), 4245--4260.
\end{thebibliography}
\end{document}